\documentclass{amsart}
\usepackage[nohug]{diagrams}
\usepackage{amscd}
\usepackage[mathscr]{eucal}
\usepackage{amssymb, amsmath, amsthm}
\usepackage{amsfonts}
\usepackage{latexsym}
\usepackage{tabularx,array}
\usepackage{hyperref}
\usepackage{latexsym}

\newfont{\msam}{msam10}

\newtheorem{theorem}[]{Theorem}
\newtheorem{proposition}[]{Proposition}
\newtheorem{corollary}[]{Corollary}
\newtheorem{lemma}[]{Lemma}

\theoremstyle{definition}
\newtheorem{definition}[]{Definition}
\newtheorem{remark}[]{Remark}

\let\nc\newcommand

\def\bthm{\begin{theorem}}
\def\ethm{\end{theorem}}
\def\blemma{\begin{lemma}}
\def\elemma{\end{lemma}}
\def\bproof{\begin{proof}}
\def\eproof{\end{proof}}
\def\bprop{\begin{proposition}}
\def\eprop{\end{proposition}}
\def\bcor{\begin{corollary}}
\def\ecor{\end{corollary}}
\nc{\la}{\label}
\def\c{\mathbb{C}}
\def\M{\mathcal{M}}

\def\L {\mathbb{L}}

\def\Com{\mathtt{Com}}

\def\Alg{\mathtt{Alg}}

\def\DGL{\mathtt{DGLA}}
\def\DGC{\mathtt{DGC}}
\def\cDGC{\mathtt{DGCC}}
\def\DGLC{\mathtt{DGLC}}

\def\DGA{\mathtt{DGA}}

\def\cDGA{\mathtt{DGCA}}

\def\DGMod{\mathtt{DG\,Mod}}

\def\D{{\mathscr D}}
\def\C{\mathcal{C}}

\def\Ho{{\mathtt{Ho}}}
\def\mfa{\mathfrak{a}}
\nc{\Ob}{{\rm Ob}}
\nc{\Hom}{{\rm{Hom}}}
\nc{\Homcont}{{\mathcal{H}om}}
\nc{\HOM}{\underline{\rm{Hom}}}
\nc{\DER}{\underline{\rm{Der}}}
\nc{\END}{\underline{\rm{End}}}
\nc{\bSym}{\mathbf{Sym}}
\nc{\Ext}{{\rm{Ext}}}
\nc{\Rep}{{\rm{Rep}}}
\nc{\DRep}{{\rm{DRep}}}
\nc{\NCRep}{\widetilde{\rm{Rep}}}
\nc{\RAct}{{\rm{RAct}}}
\nc{\bs}{\backslash}
\nc{\ob}{{\tt{Obs}}}
\nc{\CE}{\mathcal{C}}
\nc{\TP}{{T\!P}}

\nc{\nn}{{{\natural} {\natural}}}
\nc{\n}{{{\natural}}}
\nc{\A}{\mathbb A}
\nc{\B}{{\mathrm{B}}}
\nc{\Ba}{\overline{\mathrm{B}}}
\nc{\bC}{\overline{C}}
\nc{\bOmega}{\boldsymbol{\Omega}}
\nc{\bB}{\boldsymbol{B}}
\nc{\EXT}{\underline{\rm{Ext}}}
\nc{\TOR}{\underline{\rm{Tor}}}
\def\H{\mathrm H}

\def\HC{\mathrm{HC}}
\def\HR{\mathrm{HR}}
\def\rHC{\overline{\mathrm{HC}}}

\nc{\End}{{\rm{End}}}
\nc{\GL}{{\rm{GL}}}
\nc{\gl}{{\mathfrak{gl}}}
\nc{\rgl}{\overline{{\mathfrak{gl}}}}
\nc{\g}{{\mathfrak{g}}}
\nc{\h}{{\mathfrak{h}}}
\nc{\PGL}{{\rm{PGL}}}
\nc{\SL}{{\rm{SL}}}
\nc{\sll}{\mathfrak{sl}}
\nc{\cn}{ \mbox{\rm c\^{o}ne} }
\nc{\PSL}{{\rm{PSL}}}
\nc{\ad}{{\rm{ad}}}
\nc{\Ad}{{\rm{Ad}}}
\nc{\dlim}{\varinjlim}
\nc{\plim}{\varprojlim}
\nc{\colim}{{\tt{colim}}}

\newcommand{\HH}{{\rm{HH}}}

\newcommand{\Sym}{{\rm{Sym}}}

\newcommand{\id}{{\rm{Id}}}
\newcommand{\Der}{{\rm{Der}}}

\newcommand{\Tr}{{\rm{Tr}}}

\def\cb{\boldsymbol{\Omega}}

\def\bs{\backslash}

\def\U{\mathcal{U}}

\newcommand{\rar}{\xrightarrow{}}

\newcommand{\dgb}{\mathtt{DGBimod}}

\nc{\env}{\mathrm{End}(V)}
\nc{\FT}{\mathcal{C}}

\numberwithin{equation}{section}
\numberwithin{theorem}{section}
\numberwithin{lemma}{section}
\numberwithin{proposition}{section}
\numberwithin{corollary}{section}
\numberwithin{example}{section}
\numberwithin{remark}{section}
\def\ldb{\mathopen{\{\!\!\{}}
\def\rdb{\mathclose{\}\!\!\}}}

\nc{\Char}{{\rm{Ch}}}
\nc{\mfg}{\mathfrak{g}}
\nc{\CS}{{\rm{CS}}}
\nc{\bS}{\mathbb{S}}
\nc{\cO}{\mathcal{O}}
\nc{\RHom}{\rm{RHom}}
\nc{\LL}{\mathcal{L}}
\nc{\Id}{\mathrm{Id}}

\begin{document}

\begin{abstract}
There is a canonical derived Poisson structure on the universal enveloping algebra $\U\mfa$ of a (DG) Lie algebra $\mfa$ that is Koszul dual to a cyclic cocommutative (DG) coalgebra. Interesting special cases of this derived Poisson structure include (an analog of) the Chas-Sullivan bracket on string topology. We study how certain derived character of $\mfa$ intertwine this derived Poisson structure with the induced Poisson structure on the representation homology of $\mfa$. In addition, we obtain an analog of one of our main results for associative algebras.
\end{abstract}

\title{Cyclic pairings and derived Poisson structures}

\author{Ajay C. Ramadoss}
\address{Department of Mathematics,
Indiana University,
Bloomington, IN 47405, USA}
\email{ajcramad@indiana.edu}
\author{Yining Zhang}
\address{Department of Mathematics,
Indiana University,
Bloomington, IN 47405, USA}
\email{yinizhan@indiana.edu}
\maketitle

\section{Introduction}

Fix a field $k$ of characteristic $0$. In \cite{Go1}, Goldman discovered a symplectic structure on the $G$-character variety of the fundamental group $\pi$ of a Riemannian surface, where $G$ is a connected Lie group (for example, $\GL_n(\c)$). In \cite{Go2}, he further found a Lie bracket on the free $k$-vector space $k[\hat{\pi}]$ spanned by the conjugacy classes of $\pi$. It was shown that the natural trace maps
$$ \Tr_n\,:\,k[\hat{\pi}] \rar \cO[\Rep_n(\pi)]^{\GL_n} $$
are Lie algebra homomorphisms, where $\Rep_n(\pi)$ is the affine scheme parametrizing the $n$-dimensional representations of $\pi$. It was later understood that the above structure is a special case of a $\H_0$-Poisson structure on an associative algebra in the sense of \cite{CB}: such a structure on an associative algebra $A$ is given by a Lie bracket on $A/[A,A]$ satisfying some mild technical conditions. If $A$ has a $\H_0$-Poisson structure, there is an induced Poisson structure on the commutative algebra $\cO[\Rep_n(A)]^{\GL_n}$ for each $n$. Further, in this case, the canonical trace map
$$ \Tr_n\,:\,A/[A,A] \rar \cO[\Rep_n(A)]^{\GL_n} $$
is a Lie algebra homomorphism for each $n$. The notion of a $\H_0$-Poisson structure was extended in \cite{BCER} to arbitrary DG (augmented) associative algebras: a Poisson structure on $R\,\in\,\DGA_{k/k}$ is a Lie bracket on $R_\n\,:=\,R/(k+[R,R])$ satisfying the mild technical conditions referred to above. A {\it derived} Poisson structure on $A\,\in\,\DGA_{k/k}$ is a Poisson structure on some cofibrant resolution of $A$. In particular, a derived Poisson structure on $A$ gives a graded Lie bracket on the reduced cyclic homology $\rHC_\bullet(A)$ of $A$ such that the canonical higher character map,
$$ \Tr_n\,:\,\rHC_\bullet(A) \rar \HR_\bullet(A,n)^{\GL_n} $$
is a graded Lie algebra homomorphism for each $n$ (see \cite[Thm. 2]{BCER}), where $\HR_\bullet(A,n)$ stands for the representation homology parametrizing $n$-dimensional representations of $A$ (see Section \ref{SecPrelimRep} for the definition). 

When $A\,\in\,\DGA_{k/k}$ is Koszul dual to a {\it cyclic} DG coalgebra $C$, then $A$ acquires a canonical derived Poisson structure (see \cite[Thm. 15, Lem. 8]{BCER}). The general properties of such derived Poisson structures have been studied in \cite{CEEY}. In this sequel to \cite{BRZ}, we study the behaviour of certain derived character maps of a (DG) Lie algebra $\mfa$ with respect to a canonical derived Poisson structure on the universal enveloping algebra $\mathcal U\mfa$ acquired from a cyclic pairing on the Koszul dual coalgebra of $\mfa$.  In order to state our main results, we recall that a natural direct sum decomposition of the (reduced) cyclic homology $\rHC_{\bullet}(\U\mfa)$ that is Koszul dual to the Hodge (or $\lambda$-) decomposition of cyclic homology of commutative algebras was found in \cite{BFPRW} (also see \cite{Ka}):
\begin{equation} \la{IntroHodge} \rHC_\bullet(\U\mfa)\,\cong\,\bigoplus_{p=1}^{\infty} \HC_\bullet^{(p)}(\mfa)\ .\end{equation}
The direct summands of \eqref{IntroHodge} appeared in \cite{BFPRW} as domains of the {\it Drinfeld traces}. The Drinfeld trace  $\Tr_\g(P,\mfa)\,:\,\HC_\bullet^{(p)}(\mfa) \rar \HR_\bullet(\mfa,\g)$  associated with an invariant polynomial $P\,\in\,I^p(\g):=\Sym^p(\g^{\ast})^{\ad\,\g}$ is a certain derived character map with values in the representation homology of $\mfa$ in a finite dimensional Lie algebra $\g$ (see Section \ref{SecDrin} for a recapitulation of the construction). A natural interpretation of the strong Macdonald conjecture for a reductive Lie algebra $\g$ has been given in terms of the Drinfeld traces in \cite[Sec. 9]{BFPRW}. It was shown in \cite{BRZ} that if $\mfa$ is the Quillen model of a simply connected space $X$, then the summands of \eqref{IntroHodge} are the common eigenspaces of Frobenius operations on the $S^1$-equvariant homology $\H_\bullet^{S^1}(LX;k)$ of the free loop space $LX$ of $X$.

If $\mfa$ is Koszul dual to a {\it cyclic} cocommutative (DG) coalgebra $C$, then $\U\mfa$ acquires an associated derived Poisson structure. As a result, there is a Lie bracket on $\rHC_\bullet(\U\mfa)$. Such derived Poisson structures arise in topology: it is known that if $M$ is a closed simply connected manifold, there is a derived Poisson structure on $\U\mfa_M$ (where $\mfa_M$ is the Quillen model of $M$) that is associated with a cyclic pairing\footnote{This pairing is of degree $-n$, where $n=\dim M$.} on the Lambrechts-Stanley model of $M$ (see \cite[Sec. 5.5]{BCER}). The induced Lie bracket\footnote{This bracket is of degree $2-n$.} on $\rHC_{\bullet}(\U\mfa_M)\,\cong\,\overline{\H}_\bullet^{S^1}(LM;k)$ corresponds to the Chas-Sullivan bracket on
string topology. This Poisson structure on $\U\mfa_M$ induces a graded Poisson structure on $\H_{\bullet}(\U\mfa_M)\,\cong\,\H_\bullet(\Omega M;k)$ (of degree $2-\dim M$), while the Goldman bracket may be seen as a Poisson structure on $\H_\bullet(\Omega \Sigma;k)\,\cong\, k[\pi]$. The above Poisson structure on $\U\mfa_M$ may therefore 
be viewed as an analog of the Goldman bracket. In \cite{BRZ}, it was shown that the above cyclic Poisson structure on $\U\mfa$ preserves the Hodge filtration
$$ F_p\rHC_\bullet(\U\mfa)\,:=\, \bigoplus_{r \leqslant p+2} \HC_\bullet^{(r)}(\mfa)\,,$$
thus making $\rHC_\bullet(\U\mfa)$ a filtered Lie algebra. Moreover, in general,
$$ \:\{\HC_{\bullet}^{(2)}(\mfa),\HC_\bullet^{(p)}(\mfa)\}\subseteq   \HC_\bullet^{(p)}(\mfa)\,,$$
making $\HC_\bullet^{(p)}(\mfa)$ a graded Lie module over $\HC_\bullet^{(2)}(\mfa)$. If, in addition, $\g$ is reductive, there is a derived Poisson structure on the (homotopy commutative DG algebra) $\DRep_\g(\mfa)$ representing the derived scheme parametrizing the representations of $\mfa$ in $\g$. This induces a graded Poisson structure on the representation homology $\HR_\bullet(\mfa,\g)$. In the case when $\mfa=\mfa_M$, this Poisson structure is of degree $2-n$, where $n=\dim M$ (see \cite{BCER,BRZ}).

In this context, it is natural to ask how the Drinfeld traces intertwine these structures. A beginning in this direction was made in \cite{BRZ} where it was shown that the Drinfeld trace $\Tr_\g(\mfa)\,:\,\HC^{(2)}_\bullet(\mfa) \rar \HR_\bullet(\mfa,\g)$ corresponding to the Killing form is a graded Lie algebra homomorphism. The map $\Tr_\g(\mfa)$ therefore equips $\HR_\bullet(\mfa,\g)$ with the structure of a graded Lie module\footnote{In what follows, all Lie brackets as well as Lie module structures are of homological degree $n+2$, where $n$ is the degree of the cyclic pairing on the Koszul dual coalgebra.} over $\HC_\bullet^{(2)}(\mfa)$. The following theorem is our first main result.
\bthm \la{IntroThm1}
{(= Theorem \ref{intertwining})}
For any $P\,\in\,I^p(\g)$, the Drinfeld trace $\Tr_\g(P,\mfa)\,:\,\HC_\bullet^{(p)}(\mfa) \rar \HR_\bullet(\mfa,\g)$ is a homomorphism of graded $\HC^{(2)}_\bullet(\mfa)$-modules.
\ethm 

Recall that 
$$ \overline{\HH}_\bullet(\U\mfa)\,\cong\, \H_{\bullet}(\mfa;\U\mfa)\,\cong\, \bigoplus_{p=0}^{\infty}\H_\bullet(\mfa;\Sym^p(\mfa))\ ,$$
where $\mfa$ acts on $\U\mfa$ and the $\Sym^p(\mfa)$ via the adjoint action (see \cite[Thm. 3.3.2]{L}). Let $\HH^{(p)}_\bullet:=\H_\bullet(\mfa;\Sym^p(\mfa))$. It was shown in \cite{BRZ} that the Connes differential $B\,:\,\rHC_\bullet(\U\mfa) \rar \overline{\HH}_{\bullet+1}(\U\mfa)$ restricts to a map $B\,:\,\HC^{(p)}_\bullet(\mfa) \rar \HH^{(p-1)}_{\bullet+1}(\mfa)$ for all $p \geqslant 1$. In this paper, we extend the construction of the Drinfeld traces to give a map $\Tr_\g(P,\mfa)\,:\,\HH^{(p)}_{\bullet+1}(\mfa) \rar \H_\bullet[\Omega^1(\DRep_\g(\mfa))]$ for any $P\,\in\,I^{p+1}(\g)$, where $\Omega^1(\DRep_\g(\mfa))$ stands for the DG module of K\"{a}hler differentials of any (cofibrant) commutative DG algebra representing the derived affine scheme $\DRep_\g(\mfa)$. Assume that $\mfa$ is Koszul dual to a cyclic cocommutative DG coalgebra $C$. In this case, it was shown in \cite{BRZ} that $\HH^{(p)}_\bullet(\mfa)$ is a graded Lie module over $\HC^{(2)}_\bullet(\mfa)$ for all $p$, and that $B\,:\,\HC^{(p)}_\bullet(\mfa) \rar \HH^{(p-1)}_{\bullet+1}(\mfa)$ is a graded $\HC^{(2)}_\bullet(\mfa)$-module homomorphism. Our next result extends Theorem \ref{IntroThm1} as follows.
\bthm \la{IntroThm2}
{(= Theorem \ref{secondintertwining})}
For any $P\,\in\,I^{p+1}(\g)$, there is a commuting diagram of graded Lie modules over $\HC^{(2)}_\bullet(\mfa)$
$$\begin{diagram}
\HC^{(p+1)}_{\bullet}(\mfa) & \rTo^B & \HH^{(p)}_{\bullet+1}(\mfa)\\
   \dTo^{\Tr_\g(P,\mfa)}   & & \dTo^{\Tr_\g(P,\mfa)}\\
 \HR_\bullet(\mfa,\g) & \rTo^d & \H_\bullet[\Omega^1(\DRep_\g(\mfa))]\\
\end{diagram}\,,$$
where the horizontal arrow in the bottom of the above diagram is induced by the universal derivation.
\ethm

Next, we consider (augmented) associative (DG) algebras that are Koszul dual to cyclic coassociative (conilpotent, DG) coalgebras. In this case, $\rHC_{\bullet}(A)$ is a graded Lie algebra, over which $\overline{\HH}_{\bullet+1}(A)$ is a graded Lie module. In line with the Kontsevich-Rosenberg principle in noncommutative geometry, $\rHC_{\bullet}(A)$ should be seen as a {\it derived space of functions} on `$\mathrm{Spec}\,A$'. Similarly, $\overline{\HH}_{\bullet+1}(A)$ should be viewed as a {\it derived space of $1$-forms} on `$\mathrm{Spec}\,A$' (see \cite[Sec. 5]{BKR}). The Connes differential $B$ is an analog of the de Rham differential. Given that the the module of $1$-forms of a (commutative) Poisson algebra is a Lie module over that algebra itself, with the universal derivation being a Lie module homomorphism, it is natural to expect that $\overline{\HH}_{\bullet+1}(A)$ to be a graded Lie module over $\rHC_{\bullet}(A)$, with $B$ being a Lie module homomorphism. Indeed, by \cite[Thm. 1.2]{CEEY}, $B\,:\, \rHC_{\bullet}(A) \rar \overline{\HH}_{\bullet+1}(A)$ is a homomorphism of graded Lie modules over $\rHC_{\bullet}(A)$. The Kontsevich-Rosenberg principle also leads one to expect the trace maps $\Tr_n$ to induce the classical Poisson structures on $\DRep_n(A)$ and its space of $1$-forms. Confirming this expectation, we prove the following associative analog of Theorem \ref{IntroThm2}, which was stated in \cite{CEEY} (see {\it loc. cit.}, Theorem 1.3) without proof.
\bthm \la{IntroThm3}
{(= Theorem \ref{assocCD})}
There is a commutative diagram of $\rHC_{\bullet}(A)$-module homomorphisms
$$
\begin{diagram}
\rHC_{\bullet}(A) & \rTo^B & \overline{\HH}_{\bullet+1}(A)\\
 \dTo^{\Tr_n}   & & \dTo^{\Tr_n}\\
 \HR_\bullet(A,n) & \rTo^{d} & \H_{\bullet}[\Omega^1(\DRep_n(A))]\\
 \end{diagram}\ .$$
\ethm

Our final result is a common generalization of \cite[Thm. 2]{BCER} (for derived Poisson structures induced by cyclic pairings) and \cite[Thm. 5.1]{BRZ}. Following \cite{BCER,CB,G}, we define the notion of a Poisson structure for an algebra over a (finitely generated) cyclic binary quadratic operad $\mathcal P$ and show that if $A$ is a $\mathcal{P}$-algebra that is Koszul dual to a cyclic coalgebra $C$ over the (quadratic) Koszul dual operad $\mathcal Q$, then $A$ acquires a derived Poisson structure. Further, if $\mathscr{S}$ is a finite dimensional $\mathcal{P}$-algebra with a nondegenerate cyclic pairing, then the representation homology $\HR_\bullet(A,\mathscr{S})$ acquires a graded Poisson structure such that a certain canonical trace map from the $\mathcal{P}$-cyclic homology $\HC_\bullet(\mathcal P,A)$ of $A$ to $\HR_\bullet(A,\mathscr{S})$ is a graded Lie algebra homomorphism (see Theorem \ref{operadpoiss}).

\subsection*{Acknowledgements} {\footnotesize We would like to thank Yuri Berest, Ayelet Lindenstrauss, Tony Pantev and Vladimir Turaev for interesting discussions. The first author is grateful to the Department of Mathematics, University of Pennsylvania for conducive working conditions during his visit in the summer of 2018. The work of the first author was partially supported by NSF grant DMS 1702323.}

\section{Preliminaries} \la{SecPrelim}

In this section we review derived representation schemes and derived Poisson structures.
 
\subsection{Derived representation schemes} \la{SecPrelimRep}

We begin by reviewing derived representation schemes of associative and Lie algebras. Let $\DGA_{k}$ (resp., $\cDGA_{k}$) denote the category of associative (resp., commutative) DG $k$-algebras. Let $\M_n(k)$ denote the algebra of $n \times n$ matrices with entries in $k$.

\subsubsection{Associative algebras}  Consider the functor
\begin{equation} \la{repfunctor} (\mbox{--})_n\,:\, \DGA_{k/k} \rar \cDGA_{k/k} \,,\,\,\,\, A \mapsto [(A \ast_k \M_n(k))^{\M_n(k)}]_{\n\n}\,, \end{equation}
where $(A \ast_k \M_n(k))^{\M_n(k)}$ denotes the subalgebra of elements in the free product $A \ast_k \M_n(k)$ that commute with every element of $\M_n(k)$ and where $(\mbox{--})_{\n\n}$ denotes abelianization. Note that if $A$ is augmented, then $A_n$ has a natural augmentation coming from \eqref{repfunctor} applied to the augmentation map of $A$. This defines a functor $\DGA_{k/k} \rar \cDGA_{k/k}$ from the category of augmented associative DG algebras to the category of augmented commutative DG algebras, which we again denote by $(\mbox{--})_n$.

Recall that $\DGA_{k/k}$ and $\cDGA_{k/k}$ are model categories where the weak equivalences are the quasi-isomorphisms and the fibrations are the degree-wise surjections. Let $\M'_n(\mbox{--})\,:\,\cDGA_{k/k} \rar \DGA_{k/k}$ denote the functor $B \mapsto k \oplus \M_n(\bar{B})$. The functors $(\mbox{--})_n\,:\,\DGA_{k/k} \rightleftarrows \cDGA_{k/k}\,:\,\M'_n(\mbox{--})$ form a (Quillen) adjoint pair.

Thus, $A_n$ is the commutative (DG) algebra corresponding to the (DG) scheme $\Rep_n(A)$ parametrizing the $n$-dimensional representations of $A$. Since the functor $(\mbox{--})_n$ is left Quillen, it has a well behaved left derived functor
$$ \L(\mbox{--})_n\,:\,\Ho(\DGA_{k/k}) \rar \Ho(\cDGA_{k/k})\,\text{.}$$
Like for any left derived functor, we have $\L(A)_n\,\cong\,R_n$ in $\Ho(\cDGA_{k/k})$, where
$R \stackrel{\sim}{\rar} A$ is any cofibrant resolution in $\DGA_{k/k}$. We define
$$\DRep_n(A)\,:=\, \L(A)_n\,\text{ in } \Ho(\cDGA_{k/k})\,,\,\,\,\,\HR_\bullet(A,n)\,:=\,\H_\bullet[\L(A)_n]\,\text{.}$$
$\DRep_n(A)$ is called the {\it derived representation algebra} for $n$-dimensional representations of $A$. The homology $\HR_{\bullet}(A,n)$ is called the {\it representation homology} parametrizing $n$-dimensional representations of $A$. It is easy to verify that $\GL_n(k)$ acts naturally by automorphisms on the graded (commutative) algebra $\HR_\bullet(A,n)$. We denote the corresponding (graded) subalgebra of $\GL_n(k)$-invariants by $\HR_\bullet(A,n)^{\GL}$.

Let $R \stackrel{\sim}{\rar} A$ be a cofibrant resolution. The unit of the adjunction $(\mbox{--})_n\,:\,\DGA_{k/k} \rightleftarrows  \cDGA_k\,:\,\M'_n(\mbox{--})$ is the universal representation
$$ \pi_n\,:\,R \rar \M'_n(R_n) \hookrightarrow \M_n(R_n)\,\text{.}$$
It is not difficult to verify that the composite map
$$\begin{diagram} R & \rTo^{\pi_n} & \M_n(R_n) & \rTo^{\id \otimes \Tr_n} R_n \end{diagram} $$
vanishes on $[R,R]$ and that the image of the above composite map is contained in $R_n^{\GL}$. The above composite map therefore induces a map of complexes
$$ \Tr_n\,:\,R/(k+[R,R]) \rar R_n^{\GL}\,,$$
which on homologies gives the {\it derived character map}
$$\Tr_n\,:\,\rHC_\bullet(A) \rar \HR_\bullet(A,n)^{\GL}\ .$$

\subsubsection{Lie algebras} Let $\g$ be a finite dimensional Lie algebra. Consider the functor
$$(\mbox{--})_{\g} \,:\, \DGL_k \rar \cDGA_{k/k}\,, \,\,\,\,\mfa \,\mapsto \, \mfa_{\g}\,:=\, \frac{\Sym_k(\mfa \otimes \g^{\ast})}{\langle\langle ( x \otimes \xi_1).(y \otimes \xi_2) -(-1)^{|x||y|}(y \otimes \xi_1).(x \otimes \xi_2) -[x,y] \otimes \xi \rangle \rangle} \,,$$
where $\DGL_k$ is the category of DG Lie algebras over $k$, $\g^{\ast}$ is the vector space dual to $\g$ and where $\xi \mapsto \xi_1 \wedge \xi_2$ is the map dual to the Lie bracket on $\g$. The augmentation on $\mfa_\g$ is the one induced by the map taking the generators $\mfa \otimes \g^{\ast}$ to $0$. Let $\g({\mbox{--}})\,:\, \cDGA_{k/k} \rar \DGL_k$ denote the functor $B \mapsto \g(\bar{B})\,:=\, \g \otimes \bar{B}$. Recall that $\DGL_k$ is a model category where the weak-equivalences are the quasi-isomorphisms and the fibrations are the degree-wise surjections. It is shown in~\cite[Section 6.3]{BFPRW} that the functors $(\mbox{--})_{\g}\,:\, \DGL_k \rightleftarrows \cDGA_{k/k}\,:\, \g({\mbox{--}})$ form a (Quillen) adjoint pair.

Thus, $\mfa_\g$ is the commutative (DG) algebra corresponding to the (DG) scheme $\Rep_\g(\mfa)$ parametrizing representations of $\mfa$ in $\g$. Since the functor $(\mbox{--})_{\g}$ is left Quillen, it has a well behaved left derived functor
$$\L(\mbox{--})_{\g}\,:\,\Ho(\DGL_k) \rar \Ho(\cDGA_{k/k})\,\text{.}$$
Like for any left derived functor, we have  $\L(\mfa)_\g\,\cong\, \mathcal L_\g $ in $\Ho(\cDGA_{k/k})$, where $\mathcal L \stackrel{\sim}{\rar} \mfa$ is any cofibrant resolution in $\DGL_k$. We define
$$\DRep_\g(\mfa)\,:=\, \L(\mfa)_\g \, \text{ in } \Ho(\cDGA_{k/k})\,,\,\,\,\,\HR_{\bullet}(\mfa,\g)\,:=\, \H_{\bullet}[\L(\mfa)_\g]\,\text{.}$$
$\DRep_\g(\mfa)$ is called the {\it derived representation algebra} for representations of $\mfa$ in $\g$. The homology $\HR_{\bullet}(\mfa,\g)$ is called the {\it representation homology} of $\mfa$ in $\g$. It is not difficult to check that $\g$ acts naturally by derivations on the graded (commutative) algebra $\HR_{\bullet}(\mfa,\g)$. We denote the corresponding (graded) subalgebra of $\g$-invariants by $\HR_{\bullet}(\mfa,\g)^{\ad\,\g}$.

\subsection{Derived Poisson structures} The notion of a derived Poisson algebra was first introduced in \cite{BCER}, as a higher homological extension of the notion of a $\H_0$-Poisson algebra introduced by Crawley-Boevey in \cite{CB}.

\subsubsection{Definitions}
\la{Defs}
Let $ A$ be an (augmented) DG algebra. The space  $\DER(A)$  of graded $k$-linear derivations of $A$
is naturally a DG Lie algebra with respect to the commutator bracket. Let $\DER(A)^\n $ denote the subcomplex of  $\DER(A) $
comprising derivations with image in $\,k+[A,A] \subseteq A \,$. It is easy to see that $ \DER(A)^\n $ is a DG Lie ideal of $ \DER(A) $,
so that $\,\DER(A)_{\natural} := \DER(A)/\DER(A)^\n $ is a DG Lie algebra.  The natural action of $ \DER(A) $ on $A$ induces a Lie algebra
action of $ \DER(A)_\n $ on the quotient space $ A_\n := A/(k+[A,A]) $. We write $\, \varrho:  \DER(A)_\n \to \END(A_\n) \,$ for the
corresponding DG Lie algebra homomorphism.

Now,  following \cite{BCER},  we define a {\it Poisson structure} on $A$ to be  a DG Lie algebra structure on $\,A_\n \,$ such that the adjoint representation
$ \mbox{\rm ad}:\, A_\n \to \END(A_\n) $ factors through $ \varrho \,$: i.~e., there is a morphism of DG Lie algebras $\,\alpha :\, A_\n \rar \DER(A)_{\natural}\,$ such that $\,\mbox{\rm ad} = \varrho \circ \alpha \,$.  It is easy to see that if $A$ is a commutative DG algebra, then a Poisson structure on $A$ is
the same thing as a (graded) Poisson bracket on $A$. On the other hand, if $A$ is an ordinary $k$-algebra (viewed as a DG algebra), then a Poisson structure on
$A$ is precisely a ${\rm H}_0$-Poisson structure in the sense of \cite{CB}.

Let $A$ and $B$ be two Poisson DG algebras, i.e. objects of $\DGA_{k/k}$ equipped with Poisson structures.
A {\it morphism} $\,f:\, A \rar B $ of Poisson algebras is then a morphism $ f: A \to  B $ in $\DGA_{k/k} $ such that $ f_{\natural}:\, A_{\natural} \rar B_{\natural} $ is a morphism of DG Lie algebras. With this notion of morphisms, the Poisson DG algebras form a category which we denote $\mathtt{DGPA}_k $.
Note that $\mathtt{DGPA}_k $ comes  with two natural functors: the forgetful functor $ U:\, \mathtt{DGPA}_k  \to \DGA_{k/k} $ and the cyclic functor
$ (\,\mbox{--}\,)_\n : \mathtt{DGPA}_k  \to \DGL_k $. We say that a morphism  $ f $ is a {\it weak equivalence} in $ \mathtt{DGPA}_k $
if $ Uf $ is a weak equivalence in $ \DGA_{k/k} $ and $ f_\n $ is a weak equivalence in $ \DGL_k $; in other words, a weak equivalence in $\mathtt{DGPA}_k $
is a quasi-isomorphism of DG algebras,   $ f: A \to B \,$,  such that the induced map  $ f_{\n}\,:\,A_\n \rar B_\n $ is a quasi-isomorphism of DG Lie algebras.

Although we do not know at the moment whether the category $ \mathtt{DGPA}_k $ carries a Quillen model structure (with weak equivalences specified above), it has a weaker property of being a saturated homotopical category  in the sense of Dwyer-Hirschhorn-Kan-Smith \cite{DHKS} (see \cite[Sec. 3.1]{BRZ}). This allows one to define a well-behaved homotopy category of Poisson algebras and consider derived functors on $ \mathtt{DGPA}_k$: we define the homotopy category
$$ \Ho(\mathtt{DGPA}_k)\,:=\, \mathtt{DGPA}_k[\mathscr{W}^{-1}]\,,$$
where $\mathscr{W}$ is the class of weak equivalences.

\vspace{1ex}

\begin{definition}
By a {\it derived Poisson algebra} we mean a cofibrant associative DG algebra $A$ equipped with a
Poisson structure  (in the sense of Defintion~\ref{Defs}), which is viewed up to weak equivalence,
i.e. as an object in $ \Ho(\mathtt{DGPA}_k) $.
\end{definition}
Since the complex $A_{\n}$ computes the (reduced) cyclic homology of a cofibrant DG algebra $A$, the (reduced) cyclic homology of a derived Poisson algebra $A$ carries a natural structure of a graded Lie algebra (see \cite[Prop. 3.3]{BRZ}).

Another important result of \cite{BCER} that holds for the derived Poisson algebras in
$ \Ho(\mathtt{DGPA}_k) $ and that motivates our study of these objects is the following

\begin{theorem}[{\it cf.} \cite{BCER}, Theorem~2]
\la{t3s2int}
If $A$ is a derived Poisson DG algebra, then, for any $n$,
there is a unique graded Poisson bracket on the representation homology $ \HR_\bullet(A,n)^{\GL} $, such that the derived character map
$\,
\Tr_n:\, \rHC_\bullet(A) \to \HR_\bullet(A,n)^{\GL} $
is a Lie algebra homomorphism.
\end{theorem}

\subsubsection{Necklace Lie algebras}
\la{necklace}
The simplest example of a derived Poisson algebra is when $A=T_kV$, the tensor algebra generated by an even dimensional $k$-vector space $V$ equipped with a symplectic form $\langle \mbox{--},\mbox{--} \rangle\,:\, V \times V \rar V$. In this case, $A$ acquires a double Poisson structure in the sense of \cite{VdB}. The double bracket
$$\{\!\{ \mbox{--},\mbox{--}\}\!\}\,:\,\bar{A} \otimes \bar{A} \rar A \otimes A $$
is given by the formula
\begin{align} \la{ndbr}
\begin{aligned}
 &\{\!\{(v_1, \ldots, v_n), (w_1, \ldots, w_m)\}\!\}\,=\, \\
 &\sum_{\stackrel{i=1,\ldots,n}{j = 1,\ldots, m}}  \langle v_i, w_j \rangle  (w_1 ,\ldots, w_{j-1}, v_{i+1} ,\ldots, v_n)\otimes (v_1, \ldots, v_{i-1}, w_{j+1} ,\ldots, w_m) \,, \end{aligned} \end{align}
where $(v_1,\ldots,v_n)$ denotes the element $v_1 \otimes \ldots \otimes v_n \,\in\, T_kV$ for $v_1,\ldots,v_n\,\in\,V$. This double bracket can be extended to $A \otimes A$ by setting $\{\!\{a,1 \}\!\}=\{\!\{ 1,a\}\!\}=0$. The above double bracket induces a (derived) Poisson structure on $A$: the corresponding Lie bracket on $A_{\n}$ is given by the formula
$$\{ \bar{\alpha},\bar{\beta}\}\,=\, \overline{\mu \circ \{\!\{\alpha,\beta\}\!\}}\,, $$
where $\mu:A \otimes A \rar A$ is the product and where $\bar{a}$ denotes the image of $a \in A$ under the canonical projection $A \rar A_\n$. $A_\n\,=\,T_kV_{\n}$ equipped with the above Lie bracket is the well known necklace Lie algebra (see \cite{BL,G}).

\section{Koszul, Calabi-Yau algebras}

In this section, we recall results about derived Poisson structures on Koszul Calabi-Yau algebras from \cite{CEEY,BRZ}.

\subsection{Cyclic coalgebras}  \la{cyclic}

We now describe our basic construction of derived Poisson structures associated with cyclic coalgebras. Recall ({\it cf.}~\cite{GK}) that a graded associative $k$-algebra is called $n$-{\it cyclic} if it is equipped with a symmetric bilinear pairing $\langle \mbox{--},\mbox{--} \rangle\,:\, A \times A \rar k$ of degree $n$ such that
$$ \langle ab ,c \rangle \,=\,  \langle a, bc \rangle \,,\,\,\,\,\, \forall\,\, a,b,c\,\in\,A\,\text{.}$$
Dually, a graded coalgebra $C$ is called $n$-{\it cyclic} if it is equipped with a symmetric bilinear pairing $ \langle \mbox{--},\mbox{--} \rangle\,:\,  {C} \times {C} \rar k$ of degree $n$ such that
$$ \langle v', w\rangle v'' \,=\, \pm \langle v, w''\rangle w'\,,\,\,\,\,\,\, \forall\,\,v,w\,\in\,C,$$
where $v'$ and $v''$ are the components of the coproduct of $v$ written in the Sweedler notation. Note that if $A$ is a finite dimensional graded $-n$-cyclic algebra whose cyclic pairing is non-degenerate, then $C:=\Hom_k(A,k)$ is a graded $n$-cyclic coalgebra. A DG coalgebra $C$ is $n$-cyclic if it is $n$-cyclic as a graded coalgebra and
$$ \langle du, v \rangle \pm \langle u, dv \rangle \,=\, 0\,,$$
for all homogeneous $u,v \,\in\,{C}$, i.e, if $\langle \mbox{--}, \mbox{--} \rangle \,:\, {C}[n] \otimes C[n] \rar k[n]$ is a map of complexes. {\it By convention, we say that $C\,\in\,\DGC_{k/k}$ is $n$-cyclic if $\bar{C}$ is $n$-cyclic as a non-counital DG coalgebra}.

Assume that $C\,\in\,\DGC_{k/k}$ is equipped with a cyclic pairing of degree $n$ and let $R\,:=\,\cb(C)$ denote the (associative) cobar construction of $C$. Recall that $R\,\cong\, T_k(\bar{C}[-1])$  as a graded $k$-algebra. For $v_1,\ldots,v_n\,\in\,\bar{C}[-1]$, let $(v_1, \ldots ,v_n)$ denote the element $v_1 \otimes \ldots \otimes v_n$ of $R$. By~\cite[Theorem~15]{BCER}, the cyclic pairing on $C$ of degree $n$ induces a double Poisson bracket of degree $n+2$ (in the sense of~\cite{VdB})
$$\{\!\{ \mbox{--}, \mbox{--}\}\!\}\,:\, \bar{R} \otimes \bar{R} \rar {R} \otimes {R}$$
given by the formula
\begin{align} \la{dpbr}
\begin{aligned}
 &\{\!\{(v_1, \ldots, v_n), (w_1, \ldots, w_m)\}\!\}\,=\, \\
 &\sum_{\stackrel{i=1,\ldots,n}{j = 1,\ldots, m}} \pm \langle v_i, w_j \rangle  (w_1 ,\ldots, w_{j-1}, v_{i+1} ,\ldots, v_n)\otimes (v_1, \ldots, v_{i-1}, w_{j+1} ,\ldots, w_m) \,\text{.} \end{aligned} \end{align}
The above double bracket can be extended to $R \otimes R$ by setting $\{\!\{r,1\}\!\}\,=\, \{\!\{1,r\}\!\}=0$. Let $\{ \mbox{--}, \mbox{--}\}$ be the bracket associated to~\eqref{dpbr}:
\begin{equation} \la{bronr} \{ \mbox{--}, \mbox{--}\} \,:=\, \mu\,\circ\, \{\!\{ \mbox{--},\mbox{--}\}\!\}\,:\, {R} \otimes {R} \rar {R}\,,\end{equation}
 where $\mu$ is the multiplication map on ${R}$.  Let $\n\,:\, {R} \rar R_\n$ be the canonical projection and let $\{ \mbox{--}, \mbox{--}\}\,:\, \n \circ \{ \mbox{--}, \mbox{--}\}\,:\, {R} \otimes {R} \rar R_\n$. We recall that the bimodule ${R} \otimes {R}$ (with outer $R$-bimodule structure) has a double bracket (in the sense of~\cite[Defn. 3.5]{CEEY}) given by the formula
\begin{align*}
&\{\!\{ \mbox{--}, \mbox{--}\}\!\}\,\,:\,{R} \times ({R} \otimes {R}) \rar {R} \otimes ({R} \otimes {R}) \oplus ({R} \otimes {R}) \otimes {R}\,, \\
& \{\!\{r, p \otimes q\}\!\} \,:=\,  \{\!\{r,p\}\!\} \otimes q \oplus (-1)^{|p|(|r|+n)} p \otimes \{\!\{ r,q\}\!\} \,\text{.}
\end{align*}
This double bracket restricts to a double bracket on the sub-bimodule $\Omega^1R$ of $R \otimes R$ (~\cite[Corollary~5.2]{CEEY}). Let $\{\mbox{--}, \mbox{--}\}\,:\, R \otimes \Omega^1R \rar \Omega^1R$ be the map $\mu \circ \{\!\{ \mbox{--},\mbox{--}\}\!\}$, where $\mu$ is the bimodule action map and let $\{\mbox{--}, \mbox{--}\}\, :\, R \otimes \Omega^1R \rar \Omega^1R_{\n}$ denote the map $\n \circ \{ \mbox{--},\mbox{--}\}$.

 The bracket $\{ \mbox{--},\mbox{--}\}\,:\, {R} \otimes {R} \rar R_\n$ descends to a DG $(n+2)$-Poisson structure on $R$. In particular, it descends to a (DG) Lie bracket $\{\mbox{--},\mbox{--}\}_{\n}$  on $R_\n$ of degree $n+2$. The restriction of the bracket~\eqref{bronr} to $\bar{R}$ induces a degree $n+2$ DG Lie module structure over $R_\n$ on $\bar{R}$ and the bracket $\{\mbox{--},\mbox{--}\}\,:\,R \otimes \Omega^1R \rar \Omega^1R_\n$ induces a degree $n+2$ DG Lie module structure over $R_\n$ on $\Omega^1R_\n$ (see~\cite[Proposition~3.11]{CEEY}). On homologies, we have (see~
\cite{CEEY}, Theorem~1.1 and Theorem~1.2)
\bthm \la{liestronhom}
Let $A\,\in\,\DGA_{k/k}$ be an augmented associative algebra Koszul dual to $C\,\in\,\DGC_{k/k}$. Assume that $C$ is $n$-cyclic. Then, \\
$(i)$ $\rHC_{\bullet}(A)$ has the structure of a graded Lie algebra (with Lie bracket of degree $n+2$).\\
$(ii)$ $\overline{\HH}_{\bullet}(A)$ has a graded Lie module structure over $\rHC_{\bullet}(A)$ of degree $n+2$.\\
$(iii)$ The maps $S,B$ and $I$ in the Connes periodicity sequence are homomorphisms of degree $n+2$ graded Lie modules over $\rHC_{\bullet}(A)$.
\ethm
The Lie bracket of degree $n+2$ on $\rHC_{\bullet}(A)$ that is induced by a $(n+2)$-Poisson structure on $R_{\n}$ as above is an example of  a derived $(n+2)$-Poisson structure on $A$.\\

\subsubsection{} \la{convention}
\noindent
\textbf{Convention.} Since we work with algebras that are Koszul dual to $n$-cyclic coalgebras, all Lie algebras that we work with have Lie bracket of degree $n+2$. Similarly, all Lie modules are degree $n+2$ Lie modules. We therefore, drop the prefix ``degree $n+2$" in the sections that follow. Following this convention, we shall refer to (derived) $(n+2)$-Poisson structures as (derived) Poisson structures.

\subsection{Dual Hodge decomposition}

Given a Lie algebra $ \mfa $ over $k$, we consider
the symmetric ad-invariant $k$-multilinear forms on $ \mfa \,$ of a (fixed) degree $ p \ge 1 $. Every such form is induced from the universal
one: $\,\mfa \times \mfa \times \ldots \times \mfa \to \lambda^{(p)}(\mfa) \,$, which takes its values in the space $\,\lambda^{(p)}(\mfa)\,$ of coinvariants of the adjoint representation of $ \mfa $ in $ \Sym^p(\mfa)\,$.
The assignment $\,\mfa \mapsto \lambda^{(p)}(\mfa)\,$ defines a (non-additive) functor on the category of Lie algebras that extends in a canonical way to the category of DG Lie algebras:
\begin{equation}
\la{lam}
\lambda^{(p)}:\,\DGL_k \rar \Com_k \ ,\quad \mfa \mapsto \Sym^p(\mfa)/[\mfa, \Sym^p(\mfa)]\ .
\end{equation}
The category $ \DGL_k $ has a natural model structure (in the sense of Quillen \cite{Q1}), with
weak equivalences being the quasi-isomorphisms of DG Lie algebras. The corresponding homotopy (derived) category $ \Ho(\DGL_k) $ is obtained from $ \DGL_k $ by localizing at the class of
weak equivalences, i.e. by formally
inverting all the quasi-isomorphisms in $ \DGL_k $. The functor \eqref{lam}, however,
does {\it not} preserve quasi-isomorphisms and hence does not descend to the homotopy category
$ \Ho(\DGL_k) $. To remedy this problem, one has to replace $\,\lambda^{(p)}\,$ by its (left) derived functor
\begin{equation}
\la{Llam}
\L\lambda^{(p)}:\,\Ho(\DGL_k) \to \D(k)\ ,
\end{equation}
which takes its values in the derived category $ \D(k)  $
of $k$-complexes. We write  $\,\HC^{(p)}_{\bullet}(\mfa)\,$ for the homology of $\, \L\lambda^{(p)}(\mfa) \,$ and call it the {\it Lie-Hodge homology} of $ \mfa $.

For $ p = 1 $, the functor $ \lambda^{(1)} $ is just abelianization of
Lie algebras; in this case, the existence of $\, \L\lambda^{(1)} \,$ follows from
Quillen's general theory (see \cite[Chapter~II, \S 5]{Q1}), and
$\, \HC^{(1)}_{\bullet}(\mfa) \,$ coincides (up to shift in degree)
with the classical Chevalley-Eilenberg homology $ \H_\bullet(\mfa, k) $ of the Lie algebra $ \mfa $.
For $p=2$, the functor $ \lambda^{(2)} $ was introduced by Drinfeld \cite{Dr}; the existence of $ \L\lambda^{(2)} $ was established by Getzler and Kapranov \cite{GK} who
suggested that $
\HC^{(2)}_{\bullet}(\mfa) $ should be viewed as an (operadic) version of cyclic homology for Lie algebras.

Observe that each $ \lambda^{(p)} $ comes together with a natural transformation  to the composite functor
$ \U_\n := (\,\mbox{--}\,)_\n \circ \,\U:\,\DGL_k \to \DGA_{k/k} \to \Com_k $, where
$\U $
denotes the universal enveloping algebra functor on the category of (DG) Lie algebras.  The natural transformations $\,\lambda^{(p)} \to \U_\n \,$ are induced by the symmetrization maps
\begin{equation}
\la{symfun}
\Sym^p(\mfa) \to \U\mfa\ ,\quad x_1 x_2 \ldots x_p\, \mapsto\,
\frac{1}{p!}\,\sum_{\sigma \in {\mathbb S}_p}\, \pm \,x_{\sigma(1)} \cdot x_{\sigma(2)} \cdot
\ldots \cdot x_{\sigma(p)}\ ,
\end{equation}
which, by the Poincar\'e-Birkhoff-Witt Theorem, assemble to an isomorphism of  DG $\mfa$-modules
$\,\Sym_k(\mfa) \cong \U \mfa \,$. From this, it follows that   $\,\lambda^{(p)} \to \U_\n \,$ assemble to an isomorphism of functors
\begin{equation}
\la{eqv1}
\bigoplus_{p=1}^{\infty} \lambda^{(p)} \,\cong \, \U_\n \ .
\end{equation}
On the other hand, by a theorem of Feigin and Tsygan
\cite{FT} (see also \cite{BKR}), the functor $ (\,\mbox{--}\,)_\n $ has a left
derived functor $ \L(\,\mbox{--}\,)_\n:\, \Ho(\DGA_{k/k}) \to \D(k) $ that computes
the reduced cyclic homology $\,\rHC_\bullet(R)\,$ of an associative algebra
$ R \in \DGA_{k/k} $. Since $ \U $ preserves quasi-isomorphisms and maps cofibrant
DG Lie algebras to cofibrant DG associative algebras, the isomorphism \eqref{eqv1}
induces an isomorphism of derived functors from $ \Ho(\DGL_k) $ to $ \D(k) $:
\begin{equation}
\la{eqv2}
\bigoplus_{p=1}^{\infty}\, \L\lambda^{(p)}\, \cong\, \L(\,\mbox{--}\,)_\n \circ \,\U \ .
\end{equation}
At the level of homology, \eqref{eqv2} yields the direct decomposition ({\it cf.}~\cite[Theorem 7.2]{BFPRW}.

\begin{equation}
\la{hodgeds}
\rHC_{\bullet}(\U\mfa) \,\cong\,\bigoplus_{p=1}^{\infty}\, \HC^{(p)}_{\bullet}(\mfa)\ \text{.}
\end{equation}

As explained in \cite{BFPRW}, the existence of  \eqref{eqv2} is related to the fact that $ \U\mfa $ is a cocommutative Hopf algebra, and in a sense, the Lie Hodge decomposition \eqref{hodgeds} is Koszul
dual to the classical Hodge decomposition of cyclic homology for commutative algebras.

The Lie Hodge decomposition \eqref{hodgeds} also extends to (reduced) Hochschild homology (see \cite[Sec. 2.1]{BRZ}):
$$ \overline{\HH}_\bullet(\U\mfa)\,\cong\,\bigoplus_{p=0}^{\infty} \HH^{(p)}_\bullet(\mfa)\,\text{.}$$
Under Kassel's isomorphism $\overline{\HH}_\bullet(\U\mfa)\,\cong\,\H_\bullet(\mfa; \Sym(\mfa))$ (see \cite[Theorem 3.3.2]{L}), the summand $\HH^{(p)}(\mfa)$ is identified with $\H_\bullet(\mfa;\Sym^p(\mfa))$. The Connes periodity sequence for $\U\mfa$ decomposes into a direct sum of Hodge components (see \cite[Theorem 2.2]{BRZ}): the summand of Hodge degree $p$ is given by the long exact sequence
\begin{equation} \la{conneshodgep}  \begin{diagram}[small] \ldots & \rTo^S & \HC_{n-1}^{(p+1)}(\mfa) & \rTo^{B} & {\HH}_n^{(p)}(\mfa) & \rTo^I & \HC_n^{(p)}(\mfa) & \rTo^S & \HC_{n-2}^{(p+1)}(\mfa) & \rTo & \ldots \end{diagram} \, \text{.} \end{equation}
Note that $\mfa$ is Koszul dual to a cocommutative (coaugmented, conilpotent) DG coalgebra $C$ (for example, $C$ may be taken to be the Chevalley-Eilenberg coalgebra $\C(\mfa;k)$). Thus, $\U\mfa$ is Koszul dual to $C$ viewed as a coassociative DG coalgebra. When $C$ carries a cyclic pairing (of degree $n$), $\U\mfa$ acquires a derived Poisson structure, giving a Lie bracket (of degree $n+2$) on $\rHC_{\bullet}(\U\mfa)$. Further, in this case, $\overline{\HH}_\bullet(\U\mfa)$ has a graded Lie module structure over $\rHC_{\bullet}(\U\mfa)$ of degree $n+2$.  We have (see \cite[Theorems 3.3 and 3.4]{BRZ}):
\bthm \la{hdecomp}
For all $p$, the derved Poisson bracket on $\HC_\bullet(\U\mfa)$ equips the direct summand $\HC^{(p)}_\bullet(\mfa)$ with a graded Lie module structure over $\HC^{(2)}_\bullet(\mfa)$ (of degree $n+2$), i.e.,
$$\{\HC^{(2)}_\bullet(\mfa),\HC^{(p)}_{\bullet}(\mfa)\} \,\subset\,\HC^{(p)}_\bullet(\mfa)\ .$$
 Further, $\HH^{(p)}_\bullet(\mfa)$ is equipped with a graded Lie module structure over $\HC^{(2)}_\bullet(\mfa)$ (of degree $n+2$).
\ethm
There exists a cyclic cocommutative DG coalgebra Koszul dual to $\mfa$ for a large class of interesting examples: unimodular Lie algebras (of which semisimple Lie algebras are examples), Quillen models of simply connected manifolds, etc.

\subsection{Drinfeld traces and Poisson structures on representation algebras} \la{SecDrin}

\subsubsection{Drinfeld traces}
Let $\mathcal L \stackrel{\sim}{\rar} \mfa$ be a cofibrant resolution. The unit of the adjunction $(\mbox{--})_{\g}\,:\, \DGL_k \rightleftarrows \cDGA_{k/k}\,:\, \g({\mbox{--}})$ is the universal representation
$$ \pi_\g\,:\, \mathcal L \rar \g(\mathcal L_\g) \,\text{.}$$
Let $\lambda^{(p)}\,:\,\DGL_k \rar \mathtt{Com}_k$ be the functor $\mfa \mapsto \Sym^p(\mfa)/[\mfa,\Sym^p(\mfa)]$.  There is a natural map $\lambda^{(p)}[\g(\mathcal L_\g)] \rar \mathcal L_\g \otimes \lambda^{(p)}(\g)$. For $P\,\in\, I^p(\g)\,:=\,\Sym^p(\g^{\ast})^{\ad\,\g}$, evaluation at $P$ gives a linear functional $\mathrm{ev}_P$ on $\lambda^{(p)}(\g)$. One thus has the composite map
$$\begin{diagram} \lambda^{(p)}(\mathcal L) & \rTo^{\lambda^{(p)}(\pi_\g)} &  \lambda^{(p)}[\g(\mathcal L_\g)]  & \rTo&  \mathcal L_\g \otimes \lambda^{(p)}(\g) & \rTo^{\id \otimes \mathrm{ev}_P} & \mathcal L_\g \end{diagram} $$
for $P\,\in\,I^p(\g)$. On homologies, this gives the map
$$ \Tr_\g(P, \mfa)\,:\, \HC^{(p)}_{\bullet}(\mfa) \rar \HR_{\bullet}(\mfa,\g)^{\ad\,\g}\,,$$
which we call the {\it Drinfeld trace map } associated to $P$ (see~\cite[Section 7]{BFPRW} for further details regarding this construction). If $\g$ is semisimple, the Killing form is a canonical element of $I^2(\g)$. We denote the associated Drinfeld trace  by
$$\Tr_\g(\mfa)\,:\,\HC^{(2)}_{\bullet}(\mfa) \rar \HR_{\bullet}(\mfa,\g)^{\ad\,\g}\,\text{.}$$

\subsubsection{Cyclic Lie coalgebras}
Recall from \cite[Sec. 4.5]{GK} that a cyclic pairing $\langle \mbox{--},\mbox{--} \rangle$ of degree $n$ on a DG Lie algebra $\mfa$ is a symmetric, $\ad$-invariant pairing (of degree $n$) that is compatible with differential: compatibility with differential is equivalent to the assertion that $\langle \mbox{--},\mbox{--} \rangle\,:\, \mfa \otimes \mfa \rar k[-n]$ is a map of complexes.

Dually, a cyclic pairing of degree $n$ on a DG Lie coalgebra $\mathfrak{G}$ is a symmetric pairing compatible with differential satisfying
$$ x^1\langle x^2, y\rangle \,=\, \pm y^2 \langle x, y^1 \rangle \,$$
for all $x, y \,\in \,\mathfrak{G}$, where $]x[\,=\,x^1 \otimes x^2$, etc. in the Sweedler notation. It is not difficult to verify that if $\mfa$ is a finite dimensional DG Lie algebra with a non-degenerate cyclic pairing, then $\mfa^{\ast}$ is a DG Lie coalgebra with cyclic pairing (see \cite[Prop. 2.1]{Z}).

Recall that for a DG Lie coalgebra $\mathfrak{G}$, one has the Chevalley-Eilenberg algebra $\C^c(\mathfrak{G};k)$ which is the construction formally dual to the Chevalley-Eilenberg coalgebra $\C(\mfa;k)$ of a DG Lie algebra. In particular, $\C^c(\mathfrak{G};k)$ is an augmented, {\it commutative} DG algebra.
\blemma \la{poissce}
{\cite[Lemma 5.1]{BRZ}}
If $\mathfrak{G}\,\in\,\DGLC_k$ is equipped with a cyclic pairing of degree $n$, then the Chevalley-Eilenberg algebra $\C^c(\mathfrak{G};k)$ acquires a DG Poisson structure of degree $n+2$.
\elemma
Indeed,  $\C^c(\mathfrak{G};k)\,\cong\, \Sym(\mathfrak{G}[-1])$ as a graded algebra. The symmetric pairing of degree $n$ on $\mathfrak{G}$ gives a skew-symmetric pairing on $\mathfrak{G}[-1]$ of degree $n+2$. This in turn, gives the required graded Poisson structure on  $\Sym(\mathfrak{G}[-1])$. Cyclicity ensures that this structure is compatible with the differential on $\C^c(\mathfrak{G};k)$.

\subsubsection{Poisson structures} \la{prep}

Let $\mfa \,\in\,\DGL_k$ be Koszul dual to $C\,\in\,\cDGC_{k/k}$. Assume that $C$ is equipped with a cyclic pairing of degree $n$. By \cite[Theorem 6.7]{BFPRW},
$$ \DRep_\g(\mfa)\,\cong\, \C^c(\g^{\ast}(\bar{C});k)\ .$$
If $\g$ is semisimple, tensoring the cyclic pairing on $\bar{C}$ with the paring dual to the Killing form on $\g^{\ast}$ gives a cyclic pairing of degree $n$ on the DG Lie coalgebra $\g^{\ast}(\bar{C}):=\g^{\ast} \otimes \bar{C}$. It follows from Lemma \ref{poissce} that $\C^c(\g^{\ast}(\bar{C});k)$ has a DG Poisson structure of degree $n+2$. Thus, $\HR_\bullet(\mfa,\g)$ has a graded Poisson structure of degree $n+2$. By Theorem \ref{hdecomp}, $\HC^{(2)}_{\bullet}(\mfa)$ has a Lie bracket of degree $n+2$ arising from the derived Poisson structure on $\U\mfa$ corresponding to the cyclic pairing on $C$.
\bthm \la{liehomom}
{\cite[Theorem 5.1]{BRZ}}
The Drinfeld trace $\Tr_\g(\mfa)\,:\,\HC^{(2)}_\bullet(\mfa) \rar \HR_\bullet(\mfa,\g)$ corresponding to the Killing form on $\g$ is a homomorphism of graded Lie algebras.

\ethm

\section{Main results}

Throughout this section, unless stated otherwise, let $\mfa \,\in\, \DGL_k$ be Koszul dual to $C\,\in\,\cDGC_{k/k}$. Further assume that $C$ is equipped with a cyclic pairing of degree $n$. By Theorem \ref{hdecomp}, the corresponding derived Poisson structure on $\U\mfa$ equips $\rHC_\bullet(\U\mfa)$ with the structure of a graded Lie algebra (with the Lie bracket having degree $n+2$) of which $\HC^{(2)}_\bullet(\mfa)$ is a graded Lie subalgebra. Moreover, the corresponding Lie bracket equips each Lie Hodge summand $\HC^{(p)}_\bullet(\mfa)$ with the structure of a graded Lie module over $\HC^{(2)}_\bullet(\mfa)$.

\subsection{Intertwining theorem} Let $\g$ be a finite dimensional semisimple Lie algebra. By Theorem \ref{liehomom}, there is a Poisson structure  on $\HR_\bullet(\mfa,\g)$ such that the Drinfeld trace $\Tr_\g(\mfa)\,:\,\HC^{(2)}_\bullet(\mfa) \rar \HR_\bullet(\mfa,\g)$ corresponding to the Killing form on $\g$ is a homomorphism of graded Lie algebras. This equips $\HR_\bullet(\mfa,\g)$ with the structure of a graded Lie module over $\HC^{(2)}_\bullet(\mfa,\g)$.
\bthm \la{intertwining}
For any $P\,\in\,I^p(\g)$, the Drinfeld trace $\Tr_\g(P,\mfa)\,:\,\HC^{(p)}_\bullet(\mfa) \rar \HR_\bullet(\mfa,\g)$ is a homomorphism of graded Lie modules over $\HC^{(2)}_\bullet(\mfa)$.
\ethm
We recall some technicalities before proving Theorem \ref{intertwining}. Let $R\,:=\,\cb(C)\,\in\,\DGA_{k/k}$. Let $\mathcal L\,:=\,\cb_{\mathtt{Comm}}(C)\,\in\,\DGL_k$. Then $R\,\cong\,\U\mathcal L$. Recall (see Section \ref{cyclic}) that the cyclic pairing on $C$ equips $R$ with a double Poisson bracket, and therefore, a derived Poisson structure. In particular, $R_{\n}$ is a DG Lie algebra. By \cite[Prop. 3.11]{CEEY}, the bracket \eqref{bronr} equips $\bar{R}$ with the structure of a DG Lie module over $R_\n$, with $R_\n$ acting on $\bar{R}$ by derivations. The isomorphism of functors \eqref{eqv1} applied to $\mathcal L$ gives an isomorphism
\begin{equation} \la{decompforL} R_\n \,\cong\, \bigoplus_{p=1}^{\infty} \lambda^{(p)}(\mathcal L)\ . \end{equation}
By \cite[Prop. 3.4, Cor. 3.1]{BRZ},  $\lambda^{(2)}(\mathcal L)$ is a Lie subalgebra of $R_\n$, and each $\lambda^{(p)}(\mathcal L)$ is a $\lambda^{(2)}(\mathcal L)$-module. Again by \cite[Prop. 3.4, Cor. 3.1]{BRZ}, $\mathcal L$ is a $\lambda^{(2)}(\mathcal L)$-submodule of $R$, and the symmetrization map gives an isomorphism of $\lambda^{(2)}(\mathcal L)$-modules
$$ \Sym(\mathcal L) \,\cong\, R \,,$$
where the $\lambda^{(2)}(\mathcal L)$-action on $\mathcal L$ is extended to an action on $\Sym(\mathcal L)$ by derivations. By (the proof of) \cite[Thm. 6.7]{BFPRW},  $\mathcal L_\g \,=\,\C^c(\g^{\ast}(\bar{C});k)$. It follows that $\mathcal L_\g$  has a Poisson structure induced by the cyclic pairing pairing on $\g^{\ast}(\bar{C})$ obtained by tensoring the pairing on $\bar{C}$ with the Killing form on $\g$ (see Section \ref{prep}). By (the proof of) \cite[Thm. 5.1]{BRZ}, the trace $\Tr_\g(\mathcal L)\,:\,\lambda^{(2)}(\mathcal L) \rar \mathcal L_\g$ is a graded Lie algebra homomorphism. This equips $\mathcal L_\g$ with the structure of a graded Lie module over $\lambda^{(2)}(\mathcal L)$. Since $\mathcal L_\g$ is freely generated by $\g^{\ast} \otimes V$ as a graded commutative algebra, where $V:=\bar{C}[-1]$, $\Omega^1(\mathcal L_\g)\,\cong\, \mathcal L_\g \otimes \g^{\ast} \otimes V$ as a graded $\mathcal L_\g$-module. Let $\bar{\partial}\,:\,\lambda^{(2)}(\mathcal L) \rar \mathcal L \otimes V$ denote the cyclic derivative (see \cite[Lemma 6.2]{BRZ}) and let $d\,:\,\mathcal L_\g \rar \Omega^1({\mathcal L_\g})$ denote the universal derivation. The proof of Theorem \ref{intertwining} relies on the following Lemma, whose detailed proof we postpone.
\blemma \la{lderham}
{\cite[Lemma 5.2]{BRZ}}
The following diagram commutes:
$$
\begin{diagram}
\lambda^{(2)}(\mathcal{L}) & \rTo^{\bar{\partial}} & \mathcal{L} \otimes V & \rTo^{\pi_{\g} \otimes \Id} & \mathcal{L}_{\g} \otimes \g \otimes V\\
  & \rdTo_{\Tr_\g(\mathcal{L})} & & & \dTo_{\cong}  \\
& & \mathcal{L}_\g & \rTo^d & \Omega^1({\mathcal{L}_{\g}})\\
\end{diagram}
$$
Here, the vertical isomorphism on the right identifies $\g$ with $\g^{\ast}$ through the Killing form.
\elemma
\bprop \la{univrep}
The universal representation $\pi_\g\,:\,\mathcal{L} \rar \mathcal{L}_\g \otimes \g$ is a $\lambda^{(2)}(\mathcal L)$-module homomorphism, where $\lambda^{(2)}(\mathcal L)$ acts trivially on $\g$.
\eprop
\bproof

Since $\lambda^{(2)}(\mathcal L)$ acts on $R$ (resp., $\mathcal L_\g$) by derivations, it acts on the DG Lie algebras $\mathcal L$ (resp., $\mathcal L_\g \otimes \g$) by Lie derivations. Since $\mathcal L$ is freely generated as a graded Lie algebra by $V:=\bar{C}[-1]$, it suffices to verify that for any $\alpha\,\in\,\lambda^{(2)}(\mathcal L)$ and for any $u\,\in\,V$,
\begin{equation} \la{tocheck}\pi_\g(\{ \alpha, u\}) \,=\, \{ \Tr_\g(\mathcal L)(\alpha), \pi_\g(u)\}\ .\end{equation}
It follows from \eqref{dpbr} and \cite[Lemma 6.2]{BRZ} that the restriction of the action of $\lambda^{(2)}(\mathcal L)$ on $\mathcal L$ to the (graded) subspace $V$ of $\mathcal L$ is given by the composite map
$$
\begin{diagram}
\lambda^{(2)}(\mathcal{L})\otimes V & \rTo^{\bar{\partial}\otimes \id} & (\mathcal{L}\otimes V)\otimes V & \rTo & \mathcal{L}\otimes (V\otimes V) & \rTo^{\id\otimes \langle\mbox{--}\,,\,\mbox{--}\rangle} & \mathcal{L}\,,
\end{diagram}
$$
where $\bar{\partial}\,:\,\lambda^{(2)}(\mathcal L) \rar \mathcal L \otimes V$ is the cyclic derivative. Since the Poisson structure on $\mathcal L_\g$ arises from a skew symmetric pairing on $\g^{\ast} \otimes V$, the restriction of the action of $\mathcal L_\g$ on itself (via the Poisson bracket) to the (graded) subspace $\g^{\ast}  \otimes V$ is given by the composite map
$$ \begin{diagram} \mathcal L_\g \otimes \g^{\ast} \otimes V & \rTo^{d \otimes \Id}& \Omega^1({\mathcal L_\g}) \otimes \g^{\ast} \otimes V \,\cong\, \mathcal L_\g \otimes \g^{\ast} \otimes V \otimes \g^{\ast} \otimes V & \rTo^{\Id_{\mathcal L_\g} \otimes \langle \mbox{--},\mbox{--} \rangle} & \mathcal L_\g \end{diagram} \ .$$
Now, for all $v \,\in\,V$,
$$\pi_\g(v) \,=\, \sum_{\alpha} (\xi^{\ast}_{\alpha} \otimes v) \otimes \xi_{\alpha}\,\in\,\mathcal L_\g \otimes \g\,,  $$
where $\{\xi_{\alpha}\}$ is an orthonormal basis of $\g$ with respect to the Killing form and $\{\xi_{\alpha}^{\ast}\}$ is the dual basis on $\g^{\ast}$. In particular, $\pi_\g(V) \,\subset\,\g^{\ast} \otimes V \otimes \g$. Therefore, \eqref{tocheck} follows once we verify the commutativity of the following diagram:
\begin{equation} \la{tocheck2} \begin{diagram}[small]
\lambda^{(2)}(\mathcal{L}) \otimes V & \rTo^{\bar{\partial} \otimes \Id_V} & \mathcal{L} \otimes V  \otimes V & \rTo^{\pi_{\g} \otimes \Id} & \mathcal{L}_{\g} \otimes \g \otimes V \otimes V & \rTo^{\Id_{\mathcal L_\g \otimes \g} \otimes \langle \mbox{--},\mbox{--} \rangle}& \mathcal L_\g \otimes \g \\
  & \rdTo_{\Tr_\g(\mathcal{L}) \otimes \pi_\g} & & &  && \dTo^{\Id} \\
& & \mathcal{L}_\g \otimes \g^{\ast} \otimes V \otimes \g & \rTo^{d \otimes \Id_{\g^{\ast} \otimes V \otimes \g}} &  \mathcal L_\g \otimes \g^{\ast} \otimes V \otimes \g^{\ast} \otimes V \otimes \g & \rTo^{\Id_{\mathcal L_\g} \otimes \langle \mbox{--},\mbox{--} \rangle \otimes \Id_\g}& \mathcal L_\g \otimes \g \\
\end{diagram} \ .\end{equation}
By Lemma \ref{lderham}, and since the pairing on $\g^{\ast} \otimes V$ is the pairing on $V$ tensored with the (pairing dual to the) Killing form, the commutativity of \eqref{tocheck2} follows once we verify that for all $x\,\in\,\g$,
$$ \sum_{\alpha} \langle \eta(x), \xi^{\ast}_{\alpha} \rangle \xi_{\alpha}\,=\, x\,, $$
where $\eta$ denotes the identification of $\g$ with $\g^{\ast}$ via the Killing form. This is immediately seen for all $\xi_{\alpha}$, and hence, for all elements of $\g$. This completes the proof of the desired proposition.
\eproof
\bproof[Proof of Theorem \ref{intertwining}]

It follows from Proposition \ref{univrep} that the map $\Sym(\pi_\g) \,:\,\Sym(\mathcal L) \rar \Sym(\mathcal L_\g \otimes \g)$ is $\lambda^{(2)}(\mathcal L)$-equivariant, where the $\lambda^{(2)}(\mathcal L)$ action on $\mathcal L$ (resp., $\mathcal L_\g \otimes \g$) is extended to an action on $\Sym(\mathcal L)$ (resp., $\Sym(\mathcal L_\g \otimes \g)$) by derivations. In particular, for any $p$, the map $\Sym^p(\pi_\g)\,:\,\Sym^p(\mathcal L) \rar \Sym^p(\mathcal L_\g \otimes \g)$ is $\lambda^{(2)}(\mathcal L)$-equivariant. Note that $\mathcal L_\g$ acts on $\mathcal L_\g \otimes \g$ by Lie derivations and on $\mathcal L_\g \otimes \Sym(\g)$ by derivations: both actions are induced by the Poisson bracket on $\mathcal L_\g$ and the trivial $\mathcal L_\g$-action on $\g$. It can be easily verified that the canonical projection $\Sym(\mathcal L_\g \otimes \g) \rar \mathcal L_\g \otimes \Sym(\g)$ is $\mathcal L_\g$-equivariant, where the $\mathcal L_\g$-action on $\Sym(\mathcal L_\g \otimes \g)$ is obtained by extending the action on $\mathcal L_\g \otimes \g$ by derivations. Since the $\lambda^{(2)}(\mathcal L)$-action on $\mathcal L_\g$ factors through the Lie algebra homomorphism $\Tr_\g(\mathcal L)\,:\,\lambda^{(2)}(\mathcal L) \rar \mathcal L_\g$, the canonical projection $\Sym(\mathcal L_\g \otimes \g) \rar \mathcal L_\g \otimes \Sym(\g)$ is $\lambda^{(2)}(\mathcal L)$-equivariant. Hence, the composite map
\begin{equation} \la{predr} \begin{diagram} \Sym^p(\mathcal L) & \rTo^{\Sym^p(\pi_\g)}& \Sym^p(\mathcal L_\g \otimes \g) & \rTo & \mathcal L_\g \otimes \Sym^p(\g) \end{diagram}  \end{equation}
is $\lambda^{(2)}(\mathcal L)$-equivariant. It follows that the map $\lambda^{(p)}(\mathcal L) \rar \mathcal L_\g \otimes \lambda^{(p)}(\g)$, is $\lambda^{(2)}(\mathcal L)$-equivariant, since it fits into the commutative diagram
$$ \begin{diagram}
  \Sym^p(\mathcal L) & \rTo^{\eqref{predr}}& \mathcal L_\g \otimes \Sym^p(\g)\\
   \dOnto & & \dOnto\\
   \lambda^{(p)}(\mathcal L) & \rTo & \mathcal L_\g \otimes \lambda^{(p)}(\g)\\
   \end{diagram} \ .$$
 For any $P\,\in\,I^p(\g)$, the Drinfeld trace $\Tr_\g(P,\mathcal L)$ is given by composing the map $\lambda^{(p)}(\mathcal L) \rar \mathcal L_\g \otimes \lambda^{(p)}(\g)$ with evaluation at $P$. It follows that $\Tr_\g(P,\mathcal L)$ is $\lambda^{(2)}(\mathcal L)$-equivariant. Since $\mathcal L$ is a cofibrant resolution of $\mfa$, the desired theorem follows on homologies.
\eproof
We end this section with a tedious computation verifying Lemma \ref{lderham}.
\bproof[Proof of Lemma \ref{lderham}]
Note that every element in $\lambda^{(2)}(\mathcal L)$ can be expressed as a linear combination  of images of elements of the form $x \cdot w\,\in\, \Sym^2(\mathcal L)$ under the canonical projection $\Sym^2(\mathcal L) \twoheadrightarrow \lambda^{(2)}(\mathcal L)$, where $x=[v_{1},\,[v_{2},\,\cdots[v_{n-1},\,v_{n}]\cdots]]$ for some $v_1,\ldots,v_n \,\in\,V$ and $w\,\in\,V$. Here, $V:=\bar{C}[-1]$.
Observe that
\begin{eqnarray}\label{DRDtwo}
\bar{\partial} (x\cdot w)&=& \sum_{1\leqslant j \leqslant n}\pm[[v_{j+1},\,[v_{j+2},\,\cdots[v_{n-1},\,v_{n}]\cdots]],\,[[\cdots[w,\,v_{1}]\cdots,\,v_{j-2}],\,v_{j-1}]]\otimes v_{j}\nonumber\\
&&+ [v_{1},\,[v_{2},\,\cdots[v_{n-1},\,v_{n}]\cdots]]\otimes w\,\text{,}
\end{eqnarray}
where the signs are determined by the Koszul sign rule. To verify \eqref{DRDtwo}, recall that restriction of the  cyclic derivative $\bar{\partial}$ to $V^{\otimes n}$ is given by the composite map
$$ \begin{diagram} V^{\otimes n} & \rTo^{N\cdot(\mbox{--}) } & V^{\otimes n} & \rTo & V^{\otimes n-1}\otimes V \end{diagram}\,,$$
where $N=\sum_{i=0}^{n-1} \tau^{i}$ and where the last arrow is the obvious isomorphism that permutes no factors. Here, $\tau\,:\,V^{\otimes n} \rar V^{\otimes n}$ denotes the cyclic permutation $(v_1,\ldots,v_n) \mapsto \pm (v_2,\ldots, v_n, v_1)$, where the signs are given by the Koszul sign rule. Since $x \cdot w\,=\,\frac{1}{2}(xw +(-1)^{|x||w|} wx)$, and since $N$ (and hence, $\bar{\partial}$) vanish on commutators in $R:=\cb(C)$, $\bar{\partial}(x \cdot w)=\bar{\partial}(xw)$. It is easy to see that the summand of $N(xw)$ ending in $w$ is $xw$ itself. For $1 \leq i \leq n$, we note that
$$x \cdot w\,=\, \pm [v_i,z_1] \cdot z_2\,, $$
in $\lambda^{(2)}(\mathcal L)\,\subset\,R_\n$, where $z_1:= [v_{i+1},\,[v_{i+2},\,\cdots[v_{n-1},\,v_{n}]\cdots]]$ and $z_2:=[[\cdots[w,v_1]\cdots,v_{i-2}],v_{i-1}]$. It follows from \cite[Lemma A.1]{BR} that the summand of $N(x \cdot w)$ ending in $v_i$ is given by $ \pm [z_1,z_2] \otimes v_i $. This completes the verification of \eqref{DRDtwo}.

Let $\{\xi_i\}$ be an orthonomal basis of $\g$ with respect to the Killing form and let $\{\xi_i^\ast\}$ denote the dual basis of $\g^{\ast}$. Since for all $v \,\in\,V$,
$$ \pi_g(v)\,=\,\sum_i (\xi_i^{\ast} \otimes v) \otimes \xi_i\,\in\,\mathcal L_\g \otimes \g\,,$$
$$ \pi_\g(x)\,=\, \sum_{i_1,\cdots,i_n} (\xi_{i_1}^{\ast} \otimes v_1) \cdots (\xi_{i_n}^{\ast} \otimes v_n) \otimes [\xi_{i_1},\,[\xi_{i_2},\,\cdots[\xi_{i_{n-1}},\,\xi_{i_n}]\cdots]] \,\in\,\mathcal L_\g \otimes \g \ . $$
Hence,
$$\Tr_\g(\mathcal L)(x \cdot w)\,=\, \sum_{i_0,i_1,\cdots,i_n} \langle [\xi_{i_1},\,[\xi_{i_2},\,\cdots[\xi_{i_{n-1}},\,\xi_{i_n}]\cdots]] \,,\, \xi_{i_0}\rangle  (\xi_{i_1}^{\ast} \otimes v_1) \cdots (\xi_{i_n}^{\ast} \otimes v_n) (\xi_{i_0}^{\ast} \otimes w) \,,$$
where the pairing $\langle \mbox{--},\mbox{--} \rangle$ is the Killing form.  Identifying $\Omega^1({\mathcal L_\g})$ with $\mathcal L_\g \otimes \g^{\ast} \otimes V$, we have
\begin{eqnarray} \la{doftr} d\circ \Tr_\g(\mathcal L)(x \cdot w) &=& \end{eqnarray}
 $$ \sum_{i_0,i_1,\cdots,i_n} \big\{ \sum_{1 \leqslant j \leqslant n} \pm \langle [\xi_{i_1},\,[\xi_{i_2},\,\cdots[\xi_{i_{n-1}},\,\xi_{i_n}]\cdots]] \,,\, \xi_{i_0}\rangle (\xi_{i_1}^{\ast} \otimes v_1)\widehat{\underset{j}{\cdots}}  (\xi_{i_n}^{\ast} \otimes v_n) (\xi_{i_0}^{\ast} \otimes w) \otimes (\xi_{i_j}^{\ast} \otimes v_j)  $$ $$
+ \langle [\xi_{i_1},\,[\xi_{i_2},\,\cdots[\xi_{i_{n-1}},\,\xi_{i_n}]\cdots]] \,,\, \xi_{i_0}\rangle (\xi_{i_1}^{\ast} \otimes v_1) \cdots (\xi_{i_n}^{\ast} \otimes v_n) \otimes (\xi_{i_0}^{\ast} \otimes w)\big\} \ .
  $$
By \eqref{DRDtwo}, in $\mathcal L_\g \otimes \g \otimes V$,
\begin{eqnarray} \la{pigtimesid} (\pi_\g \otimes \id)(\bar{\partial}(x \cdot w))&=& \end{eqnarray}
$$ \sum_{i_0,i_1,\cdots,i_n} \big\{ \sum_{1 \leqslant j \leqslant n} \pm   (\xi_{i_1}^{\ast} \otimes v_1)\widehat{\underset{j}{\cdots}}  (\xi_{i_n}^{\ast} \otimes v_n) (\xi_{i_0}^{\ast} \otimes w) \otimes  [[\xi_{i_{j+1}},\,[\xi_{i_{j+2}},\,\cdots[\xi_{i_{n-1}},\,\xi_{i_n}]\cdots]],\,[[\cdots[\xi_{i_0},\,\xi_{i_1}]\cdots,\,\xi_{i_{j-2}}]\,,\xi_{i_{j-1}}]] \otimes v_j $$
$$+ (\xi_{i_1}^{\ast} \otimes v_1) \cdots (\xi_{i_n}^{\ast} \otimes v_n)  \otimes [\xi_{i_1},\,[\xi_{i_2},\,\cdots[\xi_{i_{n-1}},\,\xi_{i_n}]\cdots]] \otimes w \big\} \ .$$
Since $\g$ is identified with $\g^{\ast}$ via the Killing form, the desired lemma follows from \eqref{doftr} and \eqref{pigtimesid} once we verify that for $1 \leqslant j \leqslant n$,
$$ \sum_{i_0,i_1,\cdots,i_n} \langle [[\xi_{i_{j+1}},\,[\xi_{i_{j+2}},\,\cdots[\xi_{i_{n-1}},\,\xi_{i_n}]\cdots]],\,[[\cdots[\xi_{i_0},\,\xi_{i_1}]\cdots,\,\xi_{i_{j-2}}]\,,\xi_{i_{j-1}}]], \xi_{i_j} \rangle $$
$$\,=\,  \sum_{i_0,i_1,\cdots,i_n}  \langle [\xi_{i_1},\,[\xi_{i_2},\,\cdots[\xi_{i_{n-1}},\,\xi_{i_n}]\cdots]] \,,\, \xi_{i_0}\rangle\ .$$
This is immediate from the symmetry and invariance of the Killing form.
\eproof

\subsection{Traces to Hochschild homology}

Let $V:=\bar{C}[-1]$, $R:=\cb(C)$ and $\mathcal L:=\cb_{\mathtt{Comm}}(C)$ (note that $R\,\cong\,\mathcal U\mathcal L$). By assumption, $\mathcal{L} \stackrel{\sim}{\rar} \mfa$ is a cofibrant resolution of $\mfa$ in $\DGL_k$, and $R \stackrel{\sim}{\rar} \mathcal U\mfa$ is a cofibrant resolution of $\mathcal U\mfa$.

For a $R$-bimodule $M$, let $M_\n:=M/[R,M]$. Since $\Omega^1 R\,\cong\, R \otimes V \otimes R$ as graded $R$-bimodules, $\Omega^1R_\n\,\cong\, R \otimes V$ as graded vector spaces. We equip $R \otimes V \otimes R$ (resp., $R \otimes V$) with the differential inherited from $\Omega^1R$ (resp., $\Omega^1R_\n$). Let $\beta\,:\,\Omega^1 R_\n \rar \bar{R}$ denote the map given by $r \otimes v \mapsto [r, v]$. It s known that there is an isomorphism of homologies
$$ \overline{\HH}_\bullet(\mathcal U\mfa)\,\cong\,\H_\bullet[\cn(\beta\,:\,\Omega^1R_\n \rar \bar{R})]\ .$$
For $p \geqslant 1$, let $\theta^{(p)}(\mathcal L)$ denote the subcomplex $\Sym^p(\mathcal L) \otimes V$ of $R \otimes V$, where $\Sym^p(\mathcal L)$ is embedded in $R$ via the symmetrization map. It is easy to verify that $ \beta[\theta^{(p)}(\mathcal L)]\,\subset \,\Sym^p(\mathcal L)$ (see \cite[Lemma 2.1]{BRZ}). By (the proof of) \cite[Thm. 2.2]{BRZ},
$$ \HH^{(p)}_\bullet(\mfa)\,\cong\,\H_\bullet[\cn(\beta\,:\,\theta^{(p)}(\mathcal L) \rar \Sym^p(\mathcal L))]\ .$$
Let $\theta\,:\,V \rar \g \otimes \g^{\ast} \otimes V$ denote the map $v \mapsto \sum_i \xi_i \otimes \xi_i^{\ast} \otimes v$. Let $\g^{\ast}(V):=\g^{\ast} \otimes V$. Given $P\,\in\,I^{p+1}(\g)$, consider the composite map
$$ \begin{diagram}[small]\Sym^p(\mathcal L) \otimes V  & & & & & &\\
   \dTo_{\Sym^p(\pi_\g) \otimes \theta} &  & & & & &\\
    \Sym^p(\mathcal L_\g \otimes \g) \otimes \g \otimes \g^{\ast}(V) & \rTo & \mathcal L_\g \otimes \Sym^p(\g) \otimes \g \otimes \g^{\ast}(V) & \rTo & \mathcal L_\g \otimes \Sym^{p+1}(\g) \otimes \g^{\ast}(V) & \rTo^{\mathrm{ev}_P} & \mathcal L_\g \otimes \g^{\ast} (V) \end{diagram}\,,$$
where the unlabelled arrows stand for maps multiplying the obvious factors out. Identifying $\Omega^1(\mathcal L_\g)$ with $\mathcal L_\g \otimes \g^{\ast}(V)$, we see that the above composite map gives a map
\begin{equation} \la{edrintr} \Tr_\g(P,\mathcal L)\,:\,\theta^{(p)}(\mathcal L) \rar \Omega^1(\mathcal L_\g) \end{equation}
of complexes such that the following diagram commutes:
$$\begin{diagram}
\theta^{(p)}(\mathcal L) & \rTo^{\beta} & \Sym^p(\mathcal L) \\
   \dTo^{\Tr_\g(P,\mathcal L)}      & & \dTo^{\Tr_\g(P',\mathcal L)} \\
  \Omega^1(\mathcal L_\g) & \rTo^0 & \mathcal L_\g
   \end{diagram} \ .$$
Here, the vertical arrow on the right is the composition of the canonical projection $\Sym^p(\mathcal L) \twoheadrightarrow \lambda^{(p)}(\mathcal L)$ with the Drinfeld trace $\Tr_\g(P',\mathcal L)$ for any $P' \,\in\,I^p(\g)$. The Drinfeld tace therefore extends to give a composite map (independent of the choice of $P'$)
$$\Tr_\g(P,\mathcal L)\,:\,\begin{diagram} \cn(\beta\,:\,\theta^{p}(\mathcal L) \rar \Sym^p(\mathcal L))& \rTo &  \cn(\,0\,:\Omega^1(\mathcal L_\g) \rar \mathcal L_\g) & \rOnto & \Omega^1(\mathcal L_\g)[1] \end{diagram} \  $$
On homologies, we obtain a map
$$ \Tr_\g(P,\mfa)\,:\,\HH^{(p)}_{\bullet+1}(\mfa) \rar \H_\bullet[\Omega^1(\DRep_\g(\mfa))]\ .$$
The following theorem is our second main result:
\bthm \la{secondintertwining}
For any $P\,\in\,I^{p+1}(\g)$, there is a commuting diagram of graded Lie modules over $\HC^{(2)}_\bullet(\mfa)$
$$\begin{diagram}
\HC^{(p+1)}_{\bullet}(\mfa) & \rTo^B & \HH^{(p)}_{\bullet+1}(\mfa)\\
   \dTo^{\Tr_\g(P,\mfa)}   & & \dTo^{\Tr_\g(P,\mfa)}\\
 \HR_\bullet(\mfa,\g) & \rTo^d & \H_\bullet[\Omega^1(\DRep_\g(\mfa))]\\
\end{diagram}\,,$$
where the horizontal arrow in the bottom of the above diagram is induced by the universal derivation.
\ethm

\subsubsection{Proof of Theorem \ref{secondintertwining}}

Let $\Phi_\g\,:\,\Omega^1R \rar \Omega^1(\mathcal L_\g) \otimes \mathcal U\g$ denote the composite map
$$
\begin{diagram}[small]
\Omega^1(R)\,\cong\,R\otimes V\otimes R \\
\dTo_{\mathcal{U}\pi_{\g}\otimes \theta\otimes \mathcal{U}\pi_{\g}}\\
\mathcal{U}(\mathcal{L}_{\g}\otimes \g) \otimes \g\otimes \g^{\ast}(V) \otimes \mathcal{U}(\mathcal{L}_{\g}\otimes \g)\\
\dTo\\
(\mathcal{L}_{\g} \otimes \mathcal{U}\g) \otimes \g^{\ast}(V) \otimes \g \otimes (\mathcal{L}_{\g}\otimes \mathcal{U}\g)\\
\dTo\\
(\mathcal{L}_{\g}\otimes\mathcal{L}_{\g})\otimes \g^{\ast}(V)  \otimes (\mathcal{U}\g\otimes \g \otimes \mathcal{U}\g)\\
\dTo_{\mu_{\mathcal{L}_{\g}}\otimes \id \otimes \mu_{\mathcal{U}\g}}\\
\mathcal{L}_{\g}\otimes \g^{\ast}(V) \otimes \mathcal{U}\g\,\cong\,\Omega^1(\mathcal L_\g) \otimes \mathcal U\g \ .
\end{diagram}
$$
It is easy to check that $\Phi_{\g}$ is a $R$-bimodule map, where the $R$-bimodule structure on $\mathcal{L}_{\g}\otimes \g^{\ast}(V) \otimes \mathcal{U}\g$ is induced by the composite map (also denoted by $\mathcal{U}\pi_{\g}$)
$$
\begin{diagram}
R & \rTo^{\mathcal{U}\pi_{\g}} & \mathcal{U}(\mathcal{L}_{\g}\otimes \g) & \rTo & \mathcal{L}_{\g}\otimes \mathcal{U}\g\ .
\end{diagram}
$$
Recall from Section \ref{cyclic} that the $R$-bimodule $\Omega^1R$ acquires a double bracket. Composing this double bracket with the bimodule action map gives a map $R \otimes \Omega^1R \rar \Omega^1R$. By \cite[Prop. 3.10]{CEEY}, this map equips $\Omega^1R$ with the structure of a DG Lie module over $R_\n$ (and, by restriction, $\lambda^{(2)}(\mathcal L)$). On the other hand, by Lemma \ref{poissce}, $\mathcal L_\g$ acquires a DG Poisson structure such that the Drinfeld trace $\Tr_\g\,:\,\lambda^{(2)}(\mathcal L) \rar \mathcal L_\g$ corresponding to the Killing form is a homomorphism of DG Lie algebras (see \cite[Sec. 5]{BRZ}). There is a Poisson action of $\mathcal L_\g$ on $\Omega^1(\mathcal L_\g)$ making the latter a DG Lie module over the former. This equips  $\Omega^1(\mathcal L_\g)$ with the structure of a DG Lie module over $\lambda^{(2)}(\mathcal L)$. Equip $\Omega^1(\mathcal L_\g) \otimes \mathcal U\g$ with the $\lambda^{(2)}(\mathcal L)$-module structure coming from the above $\lambda^{(2)}(\mathcal L)$-action on $\Omega^1(\mathcal L_\g)$  and the trivial action on $\mathcal U\g$. The following proposition summarizes the main facts used for the proof of Theorem \ref{secondintertwining}.
\bprop \la{phig}
All maps in the commutative diagram below are $\lambda^{(2)}(\mathcal L)$-equivariant:
\begin{equation} \la{diffCD}
\begin{diagram}
R  & \rTo^d & \Omega^1R\\
 \dTo^{\mathcal U\pi_\g}  & & \dTo^{\Phi_\g}\\
 \mathcal L_\g \otimes \mathcal U\g & \rTo^{d \otimes \Id_{\mathcal U\g}} & \Omega^1(\mathcal L_\g)\otimes \mathcal U\g
\end{diagram}
\end{equation}
\eprop
The following lemma is required for the proof of Proposition \ref{phig}.
\blemma \la{raction}
The action of $R$ on $\Omega^1R\,\cong\,R \otimes V \otimes R$ is given by:
\begin{equation}\label{DBRVR1}
\{a,\,b\otimes v\otimes c\}=
\{ a,\,b\} \otimes v\otimes c
+(-1)^{(|a|+n)(|b|+|v|)}b\otimes v\otimes \{ a,\,c\}
+(-1)^{(|a|+n)|b|}b(d\{a,\,v\})c\ .
\end{equation}
\elemma
\bproof
the isomorphism $I\,:\, R\otimes V \otimes R \rightarrow \Omega_{R}^{1}\subset R\otimes R$ is given by
$$I\{(v_{1}\,\cdots\,v_{p-1})\otimes v_{p} \otimes(v_{p+1}\,\cdots\,v_{m})\}=(v_{1}\,\cdots\,v_{p})\otimes(v_{p+1}\,\cdots\,v_{m})-(v_{1}\,\cdots\,v_{p-1})\otimes(v_{p}\,\cdots\,v_{m})\ .$$
Under this isomorphism, the universal derivation $d\,:\,R \rar \Omega^1R$  is given by
$$d(v_{1}\,\cdots\,v_{m})=\sum_{i=1}^{m}(v_{1}\,\cdots\,v_{i-1})\otimes v_{i}\otimes (v_{i+1}\,\cdots\,v_{m})\,,$$
and for all $r \,\in\,R$,
$$ I(dr)\,=\,r \otimes 1 - 1 \otimes r \ .$$
For $a,\,b,\,c\in R$ and $v\in V$, consider
\begin{eqnarray*}
\ldb a,\,I(b\otimes v\otimes c)\rdb &=&\ldb a,\,bv\otimes c-b\otimes vc\rdb\\
&=&\ldb a,\,bv\rdb\otimes c+(-1)^{(|a|+n)(|b|+|v|)}bv\otimes \ldb a,\,c\rdb\\
&&-\ldb a,\,b\rdb\otimes vc-(-1)^{(|a|+n)|b|}b\otimes \ldb a,\,vc\rdb\\
&=& \ldb a,\,b\rdb v\otimes c+(-1)^{(|a|+n)|b|}b\ldb a,\,v\rdb\otimes c+(-1)^{(|a|+n)(|b|+|v|)}bv\otimes \ldb a,\,c\rdb\\
&&-\ldb a,\,b\rdb\otimes vc-(-1)^{(|a|+n)|b|}b\otimes \ldb a,\,v\rdb c-(-1)^{(|a|+n)(|b|+|v|)}b\otimes v\ldb a,\,c\rdb\\
&=& \ldb a,\,b\rdb v\otimes c-\ldb a,\,b\rdb\otimes vc\\
&& +(-1)^{(|a|+n)(|b|+|v|)}(bv\otimes \ldb a,\,c\rdb-b\otimes v\ldb a,\,c\rdb)\\
&&+(-1)^{(|a|+n)|b|}(b\ldb a,\,v\rdb\otimes c-b\otimes \{a,\,v\}\otimes c)\\
&&+(-1)^{(|a|+n)|b|}(b\otimes \{a,\,v\}\otimes c-b\otimes \ldb a,\,v\rdb c)\,.
\end{eqnarray*}

Note that both $\ldb a,\,b\rdb v\otimes c-\ldb a,\,b\rdb\otimes vc$ and $b\otimes \{a,\,v\}\otimes c-b\otimes \ldb a,\,v\rdb c$ belong to $R\otimes \Omega^1{R}$, and both $bv\otimes \ldb a,\,c\rdb-b\otimes v\ldb a,\,c\rdb$ and $b\ldb a,\,v\rdb\otimes c-b\otimes \{a,\,v\}\otimes c$ belong to $\Omega^1{R}\otimes R$. Hence,

\begin{eqnarray*}
\{a,\,I(b\otimes v\otimes c)\}\,=\, \mu \circ\ldb a,\,I(b\otimes v\otimes c)\rdb  &=& \{ a,\,b\} v\otimes c-\{ a,\,b\}\otimes vc\\
&& +(-1)^{(|a|+n)(|b|+|v|)}(bv\otimes \{ a,\,c\}-b\otimes v\{ a,\,c\})\\
&&+(-1)^{(|a|+n)|b|}(b\{a,\,v\}\otimes c-b\otimes \{a,\,v\}c)\\
& = & I(\{a,b\} \otimes v \otimes c)+(-1)^{(|a|+n)(|b|+|v|)}I(b \otimes v \otimes \{a,c\})\\
& & + (-1)^{(|a|+n)|b|}I(b d(\{a,v\}) c)\ .
\end{eqnarray*}

In the above computation, $\mu\,:\,R \otimes \Omega^1R \oplus \Omega^1R \otimes R \rar \Omega^1R$ denotes the bimodule action map. The above computation completes the proof of the desired lemma.
\eproof

For the convenience of the reader, we break Proposition \ref{phig} into the following two lemmas:
\blemma \la{dcommute}
The diagram \eqref{diffCD} commutes.
\elemma
\bproof
For $v_1,\cdots,v_n\,\in\,V$,
\begin{eqnarray*}
\Phi_{\g}\circ d(v_{1}\,\cdots\,v_{n}) &=& \Phi_{\g} \big(\sum_{i=1}^{n}(v_{1}\,\cdots\,v_{i-1})\otimes v_{i}\otimes (v_{i+1}\,\cdots\,v_{n})\big)\\
&=& \sum_{i=1}^{n} \sum_{j_{1},\,\cdots,\,j_{n}}\big(\pm(\xi_{j_1}^{\ast} \otimes v_1)\cdots \widehat{(\xi_{j_i}^{\ast} \otimes v_i)} \cdots(\xi_{j_{n}}^{\ast} \otimes v_n)\\
&&\otimes(\xi_{j_i}^{\ast} \otimes v_i)\otimes (\xi_{j_{1}}\cdots\xi_{j_{n}})\big)\ .
\end{eqnarray*}
On the other hand,
\begin{eqnarray*}
(d\otimes \id)\circ \mathcal{U}\pi_{\g}(v_{1}\,\cdots\,v_{n}) &=& (d\otimes \id)\big(\sum_{j_{1},\,\cdots,\,j_{n}}(\xi_{j_1}^{\ast} \otimes v_1)\cdots (\xi_{j_n}^{\ast} \otimes v_n)\otimes (\xi_{j_{1}}\cdots\xi_{j_{n}})\big)\\
&=& \sum_{j_{1},\,\cdots,\,j_{n}}\sum_{i=1}^{n}\big(\pm (\xi_{j_1}^{\ast} \otimes v_1)\cdots \widehat{(\xi_{j_i}^{\ast} \otimes v_i)} \cdots(\xi_{j_n}^{\ast} \otimes v_n)\\
&&\otimes(\xi_{j_i}^{\ast} \otimes v_i)\otimes (\xi_{j_{1}}\cdots\xi_{j_{n}})\big)\ .
\end{eqnarray*}
This proves the desired lemma.
\eproof
\blemma \la{equiv}
All maps in the diagram \eqref{diffCD} are $\lambda^{(2)}(\mathcal L)$-equivariant.
\elemma
\bproof
It follows immediately from Proposition \ref{univrep} that $\mathcal U\pi_\g$ is $\lambda^{(2)}(\mathcal L)$-equivariant. Since the $\lambda^{(2)}(\mathcal L)$-action on $\mathcal L_\g$ (resp., $\Omega^1(\mathcal L_g)$) factors through the Poisson action of $\mathcal L_\g$ on itself (resp., $\Omega^1(\mathcal L_g)$), the universal derivation $d\,:\,\mathcal L_\g \rar \Omega^1(\mathcal L_\g)$ is $\lambda^{(2)}(\mathcal L)$-equivariant. Indeed, the Poisson action of $\mathcal{L}_{\g}$ on $\Omega^{1}(\mathcal{L}_{\g})$ is explicitly given by
\begin{equation}
\{\eta,\,\beta d\gamma \}=\{\eta,\,\beta\}d \gamma+(-1)^{(|\eta|+n)|\beta|}\beta d\{\eta,\,\gamma\}
\end{equation}
for $\eta,\beta,\gamma \,\in\,\mathcal L_\g$, since the Poisson bracket on $\mathcal{L}_{\g}$ has degree $n+2$. The $R_{\n}$-equivariance (and therefore, $\lambda^{(2)}(\mathcal L)$-equivariance) of $d\,:\,R \rar \Omega^1R$ is a direct consequence of Lemma \ref{raction} and the fact that $R_\n$ acts by derivations on $R$. It therefore, remains to verify that $\Phi_\g$ is $\lambda^{(2)}(\mathcal L)$-equivariant. This is equivalent to the assertion that for $\alpha \,\in\,\lambda^{(2)}(\mathcal L)$, $p,q \,\in\,R$ and $u\,\in\,V$,
\begin{equation}\label{cdRVR4}
\{\Tr_{\g}(\mathcal{L})(\alpha),\,\Phi_{\g}(p\otimes u\otimes q)\}=\Phi_{\g}\big(\{\alpha,\,p\otimes u\otimes q\}\big)\ .
\end{equation}
By Lemma \ref{raction},
\begin{eqnarray}\label{DBRVR2}
&&\{\alpha,\,p\otimes u\otimes q\}\nonumber\\
&&=\{\alpha,\,p\} \otimes u\otimes q
+(-1)^{(|\alpha|+n)(|p|+|u|)}p\otimes u\otimes \{ \alpha,\,q\}
+(-1)^{(|\alpha|+n)|p|}p(d\{\alpha,\,u\})q
\end{eqnarray}
where $d\,:\,R \rightarrow R\otimes V \otimes R$ is the noncommutative de Rham differential. Since $\Phi_{\g}$ is an $R$-bimodule map, $\Phi_{\g}(p\otimes u\otimes q)=\mathcal{U}\pi_{\g}(p)\Phi_{\g}(1\otimes u \otimes 1)\mathcal{U}\pi_{\g}(q)$. Therefore,
\begin{eqnarray}\label{cdRVR3}
\Phi_{\g}\big(\{\alpha,\,p\otimes u\otimes q\}\big)&=&\mathcal{U}\pi_{\g}\big(\{\alpha,\,p\}\big)\Phi_{\g}(1\otimes u \otimes 1)\mathcal{U}\pi_{\g}(q)\nonumber\\
&&+(-1)^{(|\alpha|+n)|p|}\mathcal{U}\pi_{\g}(p)\Phi_{\g}\big(d\{\alpha,\,u\}\big)\mathcal{U}\pi_{\g}(q)\nonumber\\
&&+(-1)^{(|\alpha|+n)(|p|+|u|)}\mathcal{U}\pi_{\g}(p)\Phi_{\g}(1\otimes u \otimes 1)\mathcal{U}\pi_{\g}\big(\{\alpha,\,q\}\big)\ .
\end{eqnarray}
On the other hand,
\begin{eqnarray}\label{cdRVR2}
\{\Tr_{\g}(\mathcal{L})(\alpha),\,\Phi_{\g}(p\otimes u\otimes q)\} &=& \{\Tr_{\g}(\mathcal{L})(\alpha),\,\mathcal{U}\pi_{\g}(p)\Phi_{\g}(1\otimes u \otimes 1)\mathcal{U}\pi_{\g}(q)\}\nonumber\\
&=&\{\Tr_{\g}(\mathcal{L})(\alpha),\,\mathcal{U}\pi_{\g}(p)\}\Phi_{\g}(1\otimes u \otimes 1)\mathcal{U}\pi_{\g}(q)\nonumber\\
&&+(-1)^{(|\alpha|+n)|p|}\mathcal{U}\pi_{\g}(p)\{\Tr_{\g}(\mathcal{L})(\alpha),\,\Phi_{\g}(1\otimes u \otimes 1)\}\mathcal{U}\pi_{\g}(q)\nonumber\\
&&+(-1)^{(|\alpha|+n)(|p|+|u|)}\mathcal{U}\pi_{\g}(p)\Phi_{\g}(1\otimes u \otimes 1)\{\Tr_{\g}(\mathcal{L})(\alpha),\,\mathcal{U}\pi_{\g}(q)\}\ .
\end{eqnarray}
\noindent

\noindent
Since $\mathcal U\pi_\g$ is $\lambda^{(2)}(\mathcal L)$-equivariant,
$$\{\Tr_{\g}(\mathcal{L})(\alpha),\,\mathcal{U}\pi_{\g}(p)\}=\mathcal{U}\pi_{\g}\big(\{\alpha,\,p\}\big)$$
and
$$\{\Tr_{\g}(\mathcal{L})(\alpha),\,\mathcal{U}\pi_{\g}(q)\}=\mathcal{U}\pi_{\g}\big(\{\alpha,\,q\}\big)\,.$$

\noindent
By \eqref{cdRVR3} and \eqref{cdRVR2}, \eqref{cdRVR4} follows once we verify that for all $\alpha\,\in\,\lambda^{(2)}(\mathcal L)$ and $u\,\in\,V$,
\begin{equation}\label{cdRVR5}
 \{\Tr_{\g}(\mathcal{L})(\alpha),\,\Phi_{\g}(1\otimes u\otimes 1)\}=\Phi_{\g}\big(\{\alpha,\,1\otimes u\otimes 1\}\big)\ .
\end{equation}

\noindent
By Lemma \ref{dcommute}, $\Phi_\g(1 \otimes u \otimes 1)\,=\,(d \otimes \id_{\g})[\pi_\g(u)]$, where $u$ is viewed on the right hand side as an element of $V \subset \mathcal L$. Therefore,
\begin{eqnarray*}
 \{\Tr_{\g}(\mathcal{L})(\alpha),\,\Phi_{\g}(1\otimes u\otimes 1)\} &=&  \{\Tr_{\g}(\mathcal{L})(\alpha),\,(d \otimes \id_{\g})[\pi_\g(u)]\}\\
                                                                    &=& (d \otimes \id_{\g}) \{\Tr_{\g}(\mathcal{L})(\alpha),\pi_\g(u)\} \\
                                                                    &=& (d \otimes \id_{\g})[\pi_\g(\{\alpha, u\})] \qquad (\text{by Prop. \ref{univrep}}) \\
                                                                    &=& \Phi_\g[d\{\alpha,u\}] \qquad (\text{by Lemma \ref{dcommute}})\\
                                                                    &=& \Phi_{\g}\big(\{\alpha,\,1\otimes u\otimes 1\}) \qquad (\text{by Lemma \ref{raction}})\ .
                                                                    \end{eqnarray*}

The second equality above is because $d\,:\,\mathcal L_\g\rar \Omega^1\mathcal L_\g$ is equivariant with respect to the Poisson action of $\mathcal L_\g$. This verifies \eqref{cdRVR5}, completing the proof of the desired lemma.
\eproof

Let $\varphi_\g\,:\,\Omega^1R_\n \rar \Omega^1(\mathcal L_\g) \otimes \mathcal U\g$ denote the composite map
$$
\begin{diagram}[small]
\Omega^1R_\n\,\cong\, R\otimes V & & & &\\
\dTo^{\mathcal{U}\pi_{\g}\otimes \theta} & & & &\\
\mathcal{U}(\mathcal{L}_{\g}\otimes \g)\otimes \g^{\ast}(V) \otimes \g &\rTo & (\mathcal{L}_{\g}\otimes \mathcal{U}\g)\otimes \g^{\ast}(V) \otimes \g & \rTo &
\mathcal{L}_{\g}\otimes  \g^{\ast}(V) \otimes (\mathcal{U}\g\otimes \g) & \rTo^{\id \otimes \mu_{\mathcal{U}\g}}&
\mathcal{L}_{\g}\otimes\g^{\ast}(V) \otimes \mathcal{U}\g
\end{diagram}\ .
$$
It is easy to verify that
\begin{equation} \la{difference}(\Phi_\g - \varphi_\g \circ \n)(\Omega^1R) \,\subset \, \Omega^1(\mathcal L_\g) \otimes [\mathcal U\g,\mathcal U\g]\,=\,\Omega^1(\mathcal L_\g) \otimes [\g,\mathcal U\g]\,, \end{equation}
where $\n\,:\,\Omega^1R \rar \Omega^1R_\n$ denotes the canonical projection. Any $P\,\in\,I^{p+1}(\g)$ determines a linear functional $\mathrm{ev}_P$ on $\Sym(\g)$, which coincides with the usual evaluation map on $\Sym^{p+1}(\g)$ and vanishes on other symmetric powers of $\g$. Composing this linear functional with the inverse of the symmetrization map $\mathcal U\g \rar \Sym(\g)$ determines a linear functional on $\mathcal U\g$, which we still denote by $\mathrm{ev}_P$. Since $P$ is $\ad$-invariant, $\mathrm{ev}_P([\g,\mathcal U\g])=0$. It follows from \eqref{difference} that
\begin{equation} \la{twomaps} (\id \otimes \mathrm{ev}_P) \circ \Phi_\g \,=\,  (\id \otimes \mathrm{ev}_P) \circ \varphi_\g \circ \n\ . \end{equation}
\blemma \la{ldescent}
There is a commutative diagram of $\lambda^{(2)}(\mathcal L)$-equivariant maps
\begin{equation} \la{descent}
\begin{diagram}
R & \rTo^{\bar{\partial}} & \Omega^1R_\n\\
\dTo^{(\id \otimes \mathrm{ev}_P) \circ \mathcal U\pi_\g} & & \dTo_{ (\id \otimes \mathrm{ev}_P) \circ \varphi_\g}\\
\mathcal L_\g & \rTo^d & \Omega^1(\mathcal L_\g)
  \end{diagram}
\end{equation}
\elemma

\bproof

The commutativity of \eqref{descent} is immediate from Proposition \ref{phig}, \eqref{twomaps} and the fact that $\n \circ d\,=\,\bar{\partial}$. Since $\Phi_\g$ is $\lambda^{(2)}(\mathcal L)$-equivariant so is $(\id \otimes \mathrm{ev}_P) \circ \Phi_\g$. Since $\n\,:\,\Omega^1R \rar \Omega^1R_\n$ is surjective and $R_\n$-equivariant (and hence, $\lambda^{(2)}(\mathcal L)$-equivariant as well), the map $(\id \otimes \mathrm{ev}_P) \circ \varphi_\g\,:\,\Omega^1R_\n \rar \Omega^1(\mathcal L_\g)$ is $\lambda^{(2)}(\mathcal L)$-equivariant as well. The  $R_\n$-equivariance of $\n$ and the $R_\n$ equivariance of $d$ together imply $R_\n$-equivariance (and hence, $\lambda^{(2)}(\mathcal L)$-equivariance) of $\bar{\partial}$. Since $\mathcal U\pi_\g$ is  $\lambda^{(2)}(\mathcal L)$-equivariant by Proposition \ref{phig}, so is $(\id \otimes \mathrm{ev}_P) \circ \mathcal U\pi_\g$. That $d$ is $\lambda^{(2)}(\mathcal L)$-equivariant is part of the proof of Proposition \ref{phig}.
\eproof

\bcor \la{ptheorem}
There is a commutative diagram of $\lambda^{(2)}(\mathcal L)$-equivariant maps
\begin{equation} \la{descent2}
\begin{diagram}
\lambda^{(p+1)}(\mathcal L) & \rTo^{\bar{\partial}} & \theta^{(p)}(\mathcal L)\\
\dTo^{\Tr_\g(P,\mathcal L)} & & \dTo_{\Tr_\g(P,\mathcal L)}\\
\mathcal L_\g & \rTo^d & \Omega^1(\mathcal L_\g)
  \end{diagram}
\end{equation}
\ecor
\bproof
It is easy to verify that the restriction of $(\id \otimes \mathrm{ev}_P) \circ \varphi_\g$ to the $\lambda^{(2)}(\mathcal L)$-submodule $\theta^{(p)}(\mathcal L)$ of $\Omega^1R_\n$ coincides with $\Tr_\g(P,\mathcal L)\,:\,\theta^{(p)}(\mathcal L)\rar \Omega^1(\mathcal L_\g)$. Similarly the map $\bar{\partial}$ factors through $R_\n$: by \cite[Lemma 6.2]{BRZ}, $\bar{\partial}(\lambda^{(p+1)}(\mathcal L))\,\subset\,\theta^{(p)}(\mathcal L)$. Since $P$ is $\ad$-invariant, the map $(\id \otimes \mathrm{ev}_P) \circ \mathcal U\pi_\g$ factors through $R_\n$: the restriction of the corresponding map $R_\n \rar \mathcal L_\g$ to $\lambda^{(p+1)}(\mathcal L)$ is easily seen to coincide with $\Tr_\g(P,\mathcal L)\,:\, \lambda^{(p+1)}(\mathcal L) \rar \mathcal L_\g$. The desired statement is immediate from the above observations and Lemma \ref{ldescent}.
\eproof

Since $\beta \circ \bar{\partial}\,=\, 0$, $\bar{\partial}$ defines a map of complexes
$$\partial\,:\, \lambda^{(p+1)}(\mathcal L) \rar \cn(\beta\,:\,\theta^{(p)}(\mathcal L) \rar \Sym^p(\mathcal L))[-1] \ .$$
It follows from Corollary \ref{ptheorem}, as well as the $\lambda^{(2)}(\mathcal L)$-equivariance of $\beta$ that there is a commutative diagram of $\lambda^{(2)}(\mathcal L)$-equivariant maps
$$
\begin{diagram}
\lambda^{(p+1)}(\mathcal L) & \rTo^{{\partial}} &  \cn(\beta\,:\,\theta^{(p)}(\mathcal L) \rar \Sym^p(\mathcal L))[-1]\\
\dTo^{\Tr_\g(P,\mathcal L)} & & \dTo_{\Tr_\g(P,\mathcal L)}\\
\mathcal L_\g & \rTo^d & \Omega^1(\mathcal L_\g)
  \end{diagram} \ .$$
On homologies, this yields a commutative diagram of $\HC^{(2)}_{\bullet}(\mfa)$-equivariant maps
$$
\begin{diagram}
\HC^{(p+1)}_{\bullet}(\mfa) & \rTo^B & \HH^{(p)}_{\bullet+1}(\mfa)\\
   \dTo^{\Tr_\g(P,\mfa)}   & & \dTo^{\Tr_\g(P,\mfa)}\\
 \HR_\bullet(\mfa,\g) & \rTo^d & \H_\bullet[\Omega^1(\DRep_\g(\mfa))]\\
\end{diagram}\ .
$$
This completes the proof of Theorem \ref{secondintertwining}.

\subsection{The associative case}

For this section, let $C\,\in\,\DGC_{k/k}$ be a $d$-cyclic coassociative (coaugmented, conilpotent) DG coalgebra Koszul dual to $A\,\in\,\Alg_{k/k}$. In particular, we do not assume that $C$ is cocommutative. Throughout this section, let  $R\,:=\,\cb(C)$ and let $V:=\bar{C}[-1]$. Thus, as graded algebras, $R\,\cong\,T_kV$.

In this setting, $A$ acquires a derived Poisson structure, whence $\overline{\HC}_{\bullet}(A)$ has a graded Lie bracket of degree. Further, $\overline{\HH}_\bullet(A)$ becomes a graded Lie module over  $\overline{\HC}_{\bullet}(A)$ (see Theorem \ref{liestronhom}. By \cite[Theorem 9]{BCER}, there is a unique graded Poisson structure on $\HR_{\bullet}(A,n)^{\GL}$ such that $\Tr_n\,:\,\rHC_{\bullet}(A) \rar \HR_\bullet(A,n)^{\GL}$ is a graded Lie algebra homomorphism. As in Convention \ref{convention}, all Lie brackets and Lie actions in this section are of homological degree $d+2$.

Recall that one has the Van den Bergh functor
$$(\mbox{--})_n\,:\,\dgb(R) \rar \DGMod(R_n) \,,\qquad P \mapsto P \otimes_{R^e} \M_n(R_n)\ . $$
There is an adjunction
\begin{equation}\la{bimodadj} (\mbox{--})_n\,:\,\dgb(R) \rightleftarrows \DGMod(R_n)\,:\,\M_n(\mbox{--}) \ .\end{equation}
We continue to denote the unit of this adjunction by $\pi_n$. Explicitly, for any DG $R$-bimodule $P$, the map $\pi_n\,:\,P \rar \M_n(P_n)$ is give by the formula
$$\pi_n(p)\,=\,\sum_{1 \leqslant i,j \leqslant n} (p \otimes_{R^e} e_{ji}) \otimes_k e_{ij}\,,$$
where the $e_{ij}$'s are the elementary matrices and each $e_{ji}$ above is viewed as an element of $\M_n(R_n)$.  The composition $\Tr\circ \pi_n\,:\, P \rar P_n$ induces a map of complexes
$$ \Tr_n\,:\,P_\n\,:=\,P/[R,P] \rar P_n\ .$$
By \cite[Example 5.1]{BKR}, $(\Omega^1R)_n\,\cong\,\Omega^1R_n$. On homologies, the map $\Tr_n\,:\,\Omega^1R_{\n} \rar \Omega^1R_n$ gives
$$\Tr_n\,:\,\overline{\HH}_{\bullet+1}(A) \rar \HR_\bullet(\Omega^1A, n)\ , $$
where $\HR_\bullet(\Omega^1A,n)$ denotes the representation homology of the $A$-bimodule $\Omega^1A$ (see \cite[Thm. 5.1(b)]{BKR} for the definition). By \cite[Thm. 5.2]{BKR}, there is a commutative diagram
\begin{equation} \la{connesd}
\begin{diagram}
\rHC_{\bullet}(A) & \rTo^B & \overline{\HH}_{\bullet+1}(A)\\
 \dTo^{\Tr_n}   & & \dTo^{\Tr_n}\\
 \HR_\bullet(A,n) & \rTo^{B_n} & \HR_{\bullet}(\Omega^1A,n)\\
 \end{diagram}\,,
\end{equation}
where $B$ denotes the Connes differential and $B_n$ denotes the map induced on homologies by the universal derivation $R_n \rar \Omega^1R_n$. The following theorem, which is analogous to \cite[Theorem 5.1]{BRZ}, is a homological generalization of a special case of \cite[Prop. 7.5.1, Prop. 7.5.2 and Prop. 7.7.2]{VdB}.
\bthm \la{intertwiningassoc}
There is a DG Poisson structure on $\HR_{\bullet}(A,n)$ such that $\Tr_n\,:\,\rHC_{\bullet}(A) \rar \HR_\bullet(A,n)$ is a graded Lie algebra homomorphism.\\
\ethm

\noindent
\textbf{Remark.} Theorem \ref{intertwiningassoc} implies that the (unique) DG Poisson structure on $\HR_\bullet(A,n)^{\GL}$ in \cite[Theorem 9]{BCER} extends to all of $\HR_\bullet(A,n)$.
\bproof
Let $\bar{\M}_n^{\ast}(C)$ denote the coalgebra $k \oplus \M_n^{\ast}(k) \otimes \bar{C}$, so that $\overline{\bar{\M}_n^{\ast}(C)}\,=\,\M_n^{\ast}(k) \otimes \bar{C}$. Let $\gl_n^{\ast}(C)$ (resp., $\gl_n^{\ast}(\bar{C})$) denote $\bar{\M}_n^{\ast}(C)$ (resp., $\M_n^{\ast}(k) \otimes \bar{C}$) viewed as a Lie coalgebra. Note that $\M_n(k)$ has a nondegenerate cyclic pairing given by
$$ \langle M, N\rangle\,:=\,\Tr(MN)\ .$$
Identifying $\M_n(k)$ with $\M_n^{\ast}(k)$ via the above pairing, we obtain a cyclic pairing on $\M_n^{\ast}(k)$. Tensoring this pairing with the pairing on $\bar{C}$, we obtain a $n$-cyclic pairing on $\overline{\bar{\M}_n^{\ast}(C)}$, which can be extended to a $n$-cyclic pairing on $\bar{\M}_n^{\ast}(C)$. It is not difficult to verify that the cyclic pairing on $\overline{\bar{\M}_n^{\ast}(C)}$ is cyclic as a pairing on $\gl_n^{\ast}(\bar{C})$. By \cite[Lemma 5.1]{BRZ}, the Chevalley-Eilenberg algebra $\C^c(\gl_n^{\ast}(\bar{C});k)$ acquires a DG Poisson structure.
Since $R_n\,\cong\, \C^c(\gl_n^{\ast}(\bar{C});k)$ by \cite[Thm. 3.1]{BFPRW}, $R_n$ has a DG Poisson structure as well. That $\Tr_n\,:\,R_\n \rar R_n$ is a DG Lie algebra homomorphism is an immediate consequence of Proposition \ref{urepeq} below. The desired result then follows on homology.
\eproof

\bprop \la{urepeq}
The universal representation $\pi_n\,:\,R \rar \M_n(R_n)$ is $R_\n$-equivariant.
\eprop
\bproof
Recall that the $R_\n$ action on $\M_n(R_n)$ is obtained by composing the Poisson action of $R_n$ on $\M_n(R_n)$ with the map $\Tr_n\,:\,R_\n \rar R_n$. Since $R_n$ acts on $\M_n(R_n)$ by derivations, so does $R_\n$. Since $\pi_n$ is a homomorphism of DG algebras and since $R_\n$  acts on $R$  by derivations, it suffices to verify that for any $u \,\in V$ and for $\alpha:=(v_1,\ldots,v_k) \,\in\,R$ for $v_1,\ldots,v_k\,\in\,V$,
\begin{equation} \la{resaction} \pi_n(\{\alpha, u\})\,=\,\{\Tr_n(\overline{\alpha}), \pi_n(u)\}\ .\end{equation}
Here, $\overline{\alpha}$ denotes the image of $\alpha$ in $R_\n$ under the canonical projection. Indeed, the restriction of the $R_\n$-action on $R$ to the subspace $V$ is given by the composite map
$$ \begin{diagram} R_\n \otimes V & \rTo^{\partial \otimes \id} & R \otimes V \otimes V & \rTo^{\id \otimes \langle \mbox{--},\mbox{--} \rangle} & R \end{diagram} \ .$$
Similarly, the restriction of the Poisson action of $R_n$ on the subspace $\M^{\ast}_n(V)$ of $R_n$ is given by the composite map
$$ \begin{diagram} R_n \otimes \M^{\ast}_n(V) & \rTo^{d \otimes \id} & \Omega^1R_n \otimes \M^{\ast}_n(V)\,\cong\,R_n \otimes \M_n^{\ast}(V) \otimes \M_n^{\ast}(V) & \rTo^{\id \otimes  \langle \mbox{--},\mbox{--} \rangle} & R_n \end{diagram} \ .$$
By \cite[Lemma 5.2]{BKR}, the following diagram commutes.
$$ \begin{diagram}
R_\n & \rTo^{\partial}& \Omega^1R_\n \\
 \dTo^{\Tr_n} & & \dTo^{\Tr_n}\\
R_n & \rTo^{d} & \Omega^1R_n
\end{diagram}\ .$$
It follows that the following diagram commutes.
$$
\begin{diagram}
R_\n \otimes V & \rTo^{\partial \otimes \id}& \Omega^1R_\n \otimes V\\
   \dTo^{\Tr_n \otimes \theta_n} & & \dTo^{\Tr_n \otimes \theta_n}\\
   R_n \otimes \M_n^{\ast}(k) \otimes V \otimes \M_n(k) & \rTo^{d \otimes \id} & \Omega^1R_n \otimes \M_n^{\ast}(k) \otimes V \otimes \M_n(k)\\
   \end{diagram}\,,
$$
where $\theta\,:\, V \rar \M_n^{\ast}(k) \otimes V \otimes \M_n(k)$ denotes the restriction of $\pi_n$ to $V$. Explicitly, for $v \,\in\,V$,
$$\theta_n(v)\,=\,\sum_{1 \leqslant i,j \leqslant n} e^{\ast}_{ij} \otimes v \otimes e_{ij} \,,$$
where $\{e^{\ast}_{ij}\}$ denotes the basis of $\M^{\ast}_n(k)$ dual to the basis comprising the elementary matrices. \eqref{resaction} therefore follows once we check that the following diagram commutes:
\begin{equation} \la{trcheck}
\begin{diagram}
\Omega^1R_\n \otimes V\,\cong\, R \otimes V \otimes V & \rTo^{\id \otimes \langle \mbox{--},\mbox{--} \rangle}& R\\
 \dTo^{\Tr_n \otimes \theta_n} & & \dTo^{\pi_n}\\
\Omega^1R_n \otimes \M^{\ast}_n(k) \otimes V \otimes \M_n(k)\,\cong\,R_n \otimes (\M_n^{\ast}(k) \otimes V)^{\otimes 2} \otimes \M_n(k) & \rTo^{\id \otimes \langle \mbox{--},\mbox{--} \rangle \otimes \id} & R_n \otimes \M_n(k)\\
\end{diagram}\ .
\end{equation}
For $r \,\in\,R$ let $r_{ij}\,\in\,R_n$ denote the coefficient of the elementary matrix $e_{ij}$ in $\pi_n(r)$. Then, for $v\,\in\,V$,
$$\Tr_n(r \otimes v)\,=\,\sum_{1 \leqslant i,j \leqslant n} r_{ij} \otimes v_{ji} \,$$
where $\Omega^1R$ (resp., $\Omega^1R_n$) is identified with $R \otimes V$ (resp., $R_n \otimes \M_n^{\ast}(k) \otimes V$). Hence, for all $w \in V$,
$$ \Tr_n \otimes \theta_n(r \otimes v \otimes w)\,=\,\sum_{1 \leqslant i,j,k,l \leqslant n} r_{ij} \otimes v_{ji} \otimes w_{kl} \otimes e_{kl}\ .$$
Note that $\langle v_{ji}, w_{kl} \rangle\,=\,\langle e^{\ast}_{ji}, e^{\ast}_{kl} \rangle \langle v,w \rangle\,=\,\delta_{ik}\delta_{jl} \langle v, w \rangle$. Hence,
$$\sum_{1 \leqslant i,j \leqslant n} \langle v, w\rangle r_{ij} \otimes e_{ij}\,=\, \sum_{1 \leqslant i,j,k,l \leqslant n} \langle v_{ji} ,w_{kl} \rangle r_{ij} \otimes e_{kl}\ .  $$
This proves that the diagram \eqref{trcheck} commutes as desired.
\eproof

The following theorem was stated without proof in \cite{CEEY} (see {\it loc. cit.}, Theorem 1.3).
\bthm \la{assocCD}
There is a commutative diagram of $\rHC_{\bullet}(A)$-module homomorphisms
$$
\begin{diagram}
\rHC_{\bullet}(A) & \rTo^B & \overline{\HH}_{\bullet+1}(A)\\
 \dTo^{\Tr_n}   & & \dTo^{\Tr_n}\\
 \HR_\bullet(A,n) & \rTo^{B_n} & \HR_{\bullet}(\Omega^1A,n)\\
 \end{diagram}\ .$$
\ethm
\bproof
Since $\Tr_n \circ \n\,=\,\Tr \circ \pi_n$, where $\n\,:\,\Omega^1R \rar \Omega^1R_\n$ is the canonical projection (which is $R_\n$-equivariant), the $R_\n$-equivariance of $\Tr_n$ would follow once it is verified that $\pi_n\,:\,\Omega^1R \rar \M_n(\Omega^1R_n)$ is $R_\n$-equivariant.  By Lemma \ref{raction}, for all $\alpha\,\in\,R_\n, a, b \,\in\, R$ and for all $m \in \Omega^1R$,
  $$\{\alpha, a \cdot m \cdot b\}\,=\,\{\alpha, a\} \cdot m \cdot b \pm a \cdot \{\alpha, m\} \cdot b \pm a \cdot m \cdot \{\alpha, b\}\,,$$
 where the signs come from the Koszul sign rule. Similarly, for $x\,\in\,R_n$, $u,v\,\in\,\M_n(R_n)$ and $w\,\in\,\M_n(\Omega^1R_n)$,
 $$\{x, u \cdot w \cdot v\}\,=\, \{x, u\} \cdot w \cdot v \pm u \cdot \{x,w\} \cdot v \pm u \cdot w \cdot \{x,v\}\ . $$
 Since the $R_\n$ action on $R_n$ (resp., $\Omega^1R_n$) factors through the Poisson action of $R_n$ on $R_n$ (resp., $\Omega^1R_n$), since $\pi_n:\Omega^1R \rar \M_n(\Omega^1R_n)$ is an $R$-bimodule homomorphism, and by Proposition \ref{urepeq}, the $R_\n$-equivariance of $\pi_n:\Omega^1R \rar \M_n(R_n)$ would follow once we verify that for all $\alpha\,\in\,R_\n$ and for all $u\,\in\,V$,
 \begin{equation} \la{assoccheck} \pi_n(\{\alpha, 1 \otimes u \otimes 1\})\,=\,\{\Tr_n(\alpha), \pi_n(1 \otimes u \otimes 1)\}\ .\end{equation}
 A trivial modification of the proof of Lemma \ref{dcommute} (see also \cite[Sec. 5]{BKR}) shows that the following diagram commutes.
 \begin{equation} \la{univder}\begin{diagram}
    R & \rTo^d & \Omega^1R \\
    \dTo^{\pi_n} & & \dTo^{\pi_n}\\
    \M_n(R_n) & \rTo^{d \otimes \id_{\M_n(k)}} & \M_n(\Omega^1R_n)\\
    \end{diagram}
\end{equation}

It follows that $\pi_n(1 \otimes u \otimes 1)=(d \otimes \id_{\M_n(k)})[\pi_n(u)]$, where $u$ is viewed on the right hand side as an element of $V \subset R$. Therefore,
 \begin{eqnarray*}
 \{\Tr_n(\alpha), \pi_n(1 \otimes u \otimes 1)\} &=& \{\Tr_n(\alpha), (d \otimes \id_{\M_n(k)})[\pi_n(u)]\}\\
                                                 &=& (d \otimes \id_{\M_n(k)})\{\Tr_n(\alpha), \pi_n(u)\}\\
                                                 &=& (d \otimes \id_{\M_n(k)})[\pi_n(\{\alpha,u\})] \qquad (\text{ by Prop. \ref{urepeq}}) \\
                                                 &=& \pi_n[d(\{\alpha,u\})] \qquad (\text{ by \eqref{univder}})\\
                                                 &=& \pi_n(\{\alpha,1 \otimes u \otimes 1\}) \qquad (\text{ by Lemma \ref{raction}})\ .
 \end{eqnarray*}
The second equality above is because $d\,:\,R_n \rar \Omega^1R_n$ is equivariant with respect to the Poisson action of $R_n$. This verifies \eqref{assoccheck}, completing the proof of the desired theorem.
\eproof

\section{An operadic generalization}

Throughout this section, let $\mathcal{P}$ denote a finitely generated binary quadratic operad
and let $\mathcal{Q} = {\mathcal P}^{!} $ denote its (quadratic) Koszul dual. Let $\mathtt{DGPA}$ denote the category of DG $\mathcal{P}$-algebras and let $\mathtt{DGQC}$ denote the category of conilpotent DG $\mathcal{Q}$-coalgebras.

\subsection{Invariant bilinear forms}
\subsubsection{Definitions}
A symmetric bilinear form $B$ on $A\in \mathtt{DGPA}$ is called {\it invariant} if for all $\mu\in \mathcal{P}(2)$,
$$B(\mu(a,\,b),\,c)=(-1)^{|a||\mu|}B(a,\,\mu(b,\,c))\ \ \ \text{ for all }a,\,b,\,c\in A\,.$$

Recall that if $\mathcal P$ is a {\it cyclic} operad, there is an $\mathbb{S}_{n+1}$ action on $\mathcal{P}(n)$ for each $n \geq 0$. There is also a notion of invariant bilinear form on algebra over a cyclic operad:\\

\begin{definition}\cite[Definition 4.1]{GK}
Let $\mathcal{P}$ be a cyclic DG operad and let $A$ be a $\mathcal{P}$-algebra. A bilinear form $B$ (with values in a $k$-vector space $V$) is invariant if for all $n\geqslant 0$, the map $B_{n}\,:\,\mathcal{P}(n)\otimes A^{\otimes (n+1)} \rar V$ defined by the formula
$$B_{n}(\mu\otimes x_{1}\otimes \cdots \otimes x_{n}\otimes x_{n+1})=B(\mu(x_{1},\,\cdots,\,x_{n}),\,x_{n+1})$$
is invariant under the action of $\mathbb{S}_{n+1}$ on $\mathcal{P}(n)\otimes A^{\otimes (n+1)}$.
\end{definition}

If $\mathcal{P}$ is a cyclic binary quadratic operad, the two notions of invariant bilinear forms agree with each other (see \cite[Proposition 4.3]{GK}).

\subsubsection{Universal invariant bilinear form} Let $A_{\n}$ denote the target of the universal invariant bilinear form on $A$: this is equal to the quotient of $A\otimes A$ by the subcomplex spanned by the images of the maps
$$
\begin{diagram}
A\otimes A\otimes A & \rTo^{\mu\otimes \id-\id \otimes \mu} & A\otimes A\,,
\end{diagram}
$$
where $\mu$ runs over all (homogeneous) elements of $\mathcal{P}(2)$.

If $\mathcal{P}$ is cyclic binary quadratic, then $A_{\n}$ is equal to the quotient of $A\otimes A$ by the subcomplex spanned by the images of the maps (see \cite[Sec. 4.7]{GK}):
$$
\begin{diagram}
\mathcal{P}(n)\otimes A^{\otimes (n+1)} & \rTo^{1-\sigma} & \mathcal{P}(n)\otimes A^{\otimes (n+1)} & \rTo A\otimes A
\end{diagram}\ \ \ \ \text{ for all }n\geqslant 1 \text{ and } \sigma \in \mathbb{S}_{n+1}\,,
$$
where the last arrow is given by
$$
\mu\otimes x_{1}\otimes \cdots \otimes x_{n}\otimes x_{n+1} \mapsto \mu(x_{1},\,\cdots,\,x_{n})\otimes x_{n+1}\,,
$$
and $\sigma$ acts diagonally. This quotient is denoted in \cite{GK} by $\lambda(A)$. In particular, taking $\mu=\id\in \mathcal{P}(1)$, we see that the canonical projection from $A\otimes A$ to $A_\n$ is symmetric. Moreover, if $R=T_{\mathcal{P}}V$ the free $\mathcal{P}$-algebra generated by $V$,

\begin{proposition}\label{OpNatural}\cite[Proposition 4.9]{GK}
There is a natural isomorphism of chain complexes
$$
R_{\n}\cong \bigoplus_{n=0}^{\infty}\mathcal{P}(n)\otimes_{\mathbb{S}_{n+1}}V^{(n+1)}\,.
$$
\end{proposition}

\begin{proof}
If $\mu\in\mathcal{P}(n)$ and $\nu\in \mathcal{P}(m)$, and $v_{i},\,w_{j}\in V$, the image of the element
$$\mu(v_{1},\,\cdots,\,v_{n})\otimes \nu(w_{1},\,\cdots,\,w_{m})\in A_{\n}$$
is equal to the image of
$$\pm(\tau^{-1}\mu)(v_{2},\,\cdots,\,v_{n},\,\nu(w_{1},\,\cdots,\,w_{m}))\otimes v_{1}$$
which is the same as
$$\pm ((\tau^{-1}\mu)\circ_{n}\nu)(v_{2},\,\cdots,\,v_{n},\,w_{1},\,\cdots,\,w_{m})\otimes v_{1}\,.$$
Thus, one can see that $R_{\n}$ is a quotient of
$$T_{\mathcal{P}}V\otimes V\cong \bigoplus_{n=0}^{\infty}\mathcal{P}(n)\otimes_{\mathbb{S}_{n}}V^{n}\otimes V\,.$$
Using once more the coinvariance under the action of the cyclic group $\mathbb{Z}_{n+1}$, it can be shown that there is an isomorphism
$$
R_{\n}\cong \bigoplus_{n=0}^{\infty}\mathcal{P}(n)\otimes_{\mathbb{S}_{n+1}}V^{(n+1)}\,.
$$
\end{proof}

\begin{remark}
By the above proof, every element in $R_{\n}$ has a representative in the form of $(\mu\otimes v_{1}\otimes \cdots \otimes v_{n})\otimes w$ for some $n\geqslant 0$ and $\mu\in \mathcal{P}(n)$, where $v_{i},\,w\in V$.
\end{remark}

The category $\mathtt{DGPA}$ is a model category where the fibrations are the degreewise surjections and the weak equivalences are the quasi-isomorphisms. By \cite[Thm. A.5]{BFPRW}, the functor $(\mbox{--})_\n\,:\,\mathtt{DGPA} \rar \Com_k\,,\,A \mapsto A_\n$ has a (total) left derived functor $\L(\mbox{--})_\n\,:\,\Ho(\mathtt{DGPA}) \rar \Ho(\Com_k)\,,\,A \mapsto R_\n$, where $R \stackrel{\sim}{\rar} A$ is any cofibrant resolution of $A$ in $\mathtt{DGPA}$. We denote the homology $\H_\bullet[\L A_\n]$ by $\HC_\bullet(\mathcal{P},A)$ and call it the {\it $\mathcal P$-cyclic homology} of $A$. When $\mathcal{P}$ is the Lie operad and when $A=\mfa$, $\HC_\bullet(\mathcal P,A)\,=\,\HC^{(2)}_{\bullet}(\mfa)$.

\subsubsection{Invariant 2-tensor}
Dually, if $\mathfrak{C}$ is a DG $\mathcal{P}$-coalgebra, a (homogeneous) 2-tensor $\alpha$ is called {\it invariant} if it is a (homogeneous) element in $\mathfrak{C}\otimes \mathfrak{C}$ such that for all $\mu\in \mathcal{P}(2)$, the composite map of complexes
$$
\begin{diagram}
k[|\alpha|] & \rTo^{\alpha} & \mathfrak{C}\otimes \mathfrak{C} & \rTo^{\mu\otimes \id-\id \otimes \mu} & \mathfrak{C}\otimes \mathfrak{C}\otimes \mathfrak{C}
\end{diagram}
$$
vanishes. We denote the subcomplex of invariant 2-tensors on $\mathfrak{C}$ by $\mathfrak{C}^{\n}$. Notice that a non-degenerate invariant bilinear form on a DG $\mathcal{P}$-algebra $A$ induces a invariant 2-tensor on its linear dual DG $\mathcal{P}$-coalgebra $A^{\ast}$.

\subsection{Derived representation schemes}
Let $\mathfrak{C}$ be a DG $\mathcal{P}$-coalgebra. Note that the complex $\mathbf{Hom}(\mathfrak{C}, \bar{B})$ naturally acquires the structure of a DG $\mathcal{P}$-algebra for any $B\,\in\,\cDGA_{k/k}$. This gives a functor
$$ \mathbf{Hom}_{\ast}(\mathfrak{C}, \mbox{--})\,:\,\cDGA_{k/k} \rar \mathtt{DGPA}\,,\qquad B \mapsto \mathbf{Hom}(\mathfrak{C}, \bar{B})\ .$$
By \cite[Prop. A.1]{BFPRW}, the functor $ \mathbf{Hom}_{\ast}(\mathfrak{C}, \mbox{--})$ has a left adjoint $\mathfrak{C} \ltimes \mbox{--}\,:\,\mathtt{DGPA} \rar \cDGA_{k/k}$. For $A\,\in\mathtt{DGPA}$, let $\pi_A\,:\,A \rar \mathbf{Hom}(\mathfrak{C}, \mathfrak{C} \ltimes A)$ denote the composite map
$$ \begin{diagram} A & \rTo & \mathbf{Hom}_\ast(\mathfrak{C}, \mathfrak{C} \ltimes A) & \rInto & \mathbf{Hom}(\mathfrak{C}, \mathfrak{C} \ltimes A)\end{diagram}\,,$$
where the first arrow is the unit of the adjunction described above. By \cite[Thm. A.2]{BFPRW}, the functor $\mathfrak{C} \ltimes \mbox{--}$ has a (total) left derived functor $\mathfrak{C} \ltimes^{\L} \mbox{--}\,:\,\Ho(\mathtt{DGPA}) \rar \Ho(\cDGA_{k/k})$. For $A\,\in\,\mathtt{DGPA}$, we call $\mathfrak{C} \ltimes^{\L} A$ the {\it derived representation scheme of $A$ over $\mathfrak{C}$} and denote it by $\DRep_{\mathfrak{C}}(A)$. The {\it representation homology} of $A$ over $\mathfrak{C}$ is the homology
$$ \HR_\bullet(A,\mathfrak{C})\,:=\,\H_\bullet[\DRep_{\mathfrak{C}}(A)]\ .$$
If $\mathfrak{C}$ is finite dimensional, $\mathscr S:=\mathfrak{C}^{\ast}$ is a $\mathcal P$-algebra and we write $\HR_\bullet(A, \mathscr S)$ (resp., $\DRep_{\mathscr S}(A)$) for $\HR_\bullet(A,\mathfrak{C})$ (resp., $\DRep_{\mathfrak{C}}(A)$).

\subsubsection{Traces}
Assume that $\mathfrak{C}$ is equipped with a degree $0$ invariant $2$-tensor %

$$ \mathrm{coTr}\,:\, k \rar \mathfrak{C}^{\n}\ .$$
Abusing notation, let $\pi_{A \otimes A}$ denote the composite map
$$ \begin{diagram} A \otimes A & \rTo^{\pi_A^{\otimes 2}} & \mathbf{Hom}(\mathfrak{C}, \mathfrak{C} \ltimes A)^{\otimes 2} \,\cong\,\mathbf{Hom}(\mathfrak{C} \otimes \mathfrak{C}, (\mathfrak{C} \ltimes A)^{\otimes 2}) & \rTo & \mathbf{Hom}(\mathfrak{C} \otimes \mathfrak{C},\mathfrak{C} \ltimes A)\end{diagram}\,,$$
where the last arrow is induced by the product on $\mathfrak{C} \ltimes A$. The map of complexes
$$\begin{diagram} A \otimes A & \rTo^{\pi_{A \otimes A}} & \mathbf{Hom}(\mathfrak{C} \otimes \mathfrak{C},\mathfrak{C} \ltimes A) & \rTo & \mathbf{Hom}(\mathfrak{C}^{\n},\mathfrak{C} \ltimes A)\end{diagram}$$
induced by the inclusion $\mathfrak{C}^{\n} \hookrightarrow \mathfrak{C} \otimes \mathfrak{C}$ factors through $A_\n$, giving a map of complexes
$$ A_\n \rar \mathbf{Hom}(\mathfrak{C}^{\n},\mathfrak{C} \ltimes A)\ .$$
Composing this with the map $\mathbf{Hom}(\mathfrak{C}^\n,\mathfrak{C} \ltimes A) \rar \mathbf{Hom}(k, \mathfrak{C} \ltimes A)$ induced by $\mathrm{coTr}$, we obtain map
$$ \Tr\,:\,A_\n \rar \mathfrak{C} \ltimes A\ .$$
By construction, $\Tr$ gives a natural transformation of functors $(\mbox{--})_\n \rar \mathfrak{C} \ltimes \mbox{--}$ (which, of course, depends on the choice of $\mathrm{coTr}$). As a result, this gives a natural transformation of {\it derived functors}
$$ \L (\mbox{--})_\n \rar \mathfrak{C} \ltimes^{\L} \mbox{--}\,:\,\Ho(\mathtt{DGPA}) \rar \Ho(\Com_k) \ . $$
Applying this to $A\,\in\,\mathtt{DGPA}$ and taking homologies, we obtain maps of graded vector spaces
$$\Tr\,:\,\HC_\bullet(\mathcal{P}, A)\rar \HR_\bullet(A,\mathfrak{C}) \ .$$
If $\mathcal{P}$ is the Lie operad, $\mathfrak{C}=\g^{\ast}$ for $\g$ semisimple and $\mathrm{coTr}$ is the Killing form, then the above construction recovers the Drinfeld trace associated with the Killing form.

\subsubsection{Poisson structures on representation homology}
We say that $\mathfrak{C}$ is {\it cyclic} if it is equipped with a symmetric bilinear form $\langle \mbox{--}\,,\,\mbox{--} \rangle\,:\, \mathfrak{C} \times \mathfrak{C} \rar k$ such that for all $\mu\in\mathcal{P}(2)$, the following diagram commutes
\begin{equation}\la{OpCCD}
\begin{diagram}[small]
\mathfrak{C} \otimes \mathfrak{C} &\rTo^{\mu\otimes \id} & \mathfrak{C}\otimes \mathfrak{C}\otimes \mathfrak{C}\\
\dTo^{\id \otimes \mu} & & \dTo_{\id \otimes \langle \mbox{--}\,,\,\mbox{--} \rangle}\\
\mathfrak{C}\otimes \mathfrak{C}\otimes \mathfrak{C} & \rTo^{\langle \mbox{--}\,,\,\mbox{--} \rangle\otimes \id} & \mathfrak{C}
\end{diagram}
\end{equation}

\begin{remark}
If the pairing $\langle \mbox{--}\,,\,\mbox{--} \rangle$ is also non-degenerate, then it induces an isomorphism between $\mathfrak{C}$ and its linear dual $\mathfrak{C}^{\ast}$. In this case, it also induces an invariant bilinear form on $\mathfrak{C}^{\ast}$ under the above isomorphism.
\end{remark}

Let $\cb_{\mathcal{Q}}\,:\,\mathtt{DGQC} \rar \mathtt{DGPA}$ denote the cobar construction (see \cite[Sec. 11.2.5]{LV}). For $C\in \mathtt{DGQC}$, the underlying graded algebra of $\cb_{\mathcal{Q}}(C)$ is the free $\mathcal{P}$-algebra generated by $C[-1]$, and the differential on $\cb_{\mathcal{Q}}(C)$ is the sum of two degree $-1$ derivations, one induced by the differential on $C$ and the other induced by operations from $\mathcal{Q}$. We say that $C\,\in\,\mathtt{DGQC}$ is {\it Koszul dual} to $A\,\in\,\mathtt{DGPA}$ if there is a quasi-isomorphism $\cb_{\mathcal{Q}}(C) \stackrel{\sim}{\rar} A$.

\begin{proposition}\la{OpLCBPoiss}
Let $A\,\in\,\mathtt{DGPA}$ be Koszul dual to $C\,\in\,\mathtt{DGQC}$. If $C$ is cyclic, and if $\mathfrak{C}\,\in\,\mathtt{DGPA}$ is cyclic, then $\DRep_{\mathfrak{C}}(A)$ acquires a derived Poisson structure. As a consequence, $\HR_\bullet(A,\mathfrak{C})$ acquires a graded Poisson structure.
\end{proposition}

\bproof

 Note that $\mathfrak{C} \otimes C$ is a conilpotent DG coalgebra over the Hadamard product $P \otimes_{\H} Q$. Further, it is not difficult to verify that if $\mathfrak{C}$ (resp., $C$) is a cyclic DG $\mathcal{P}$-coalgebra (resp., cyclic DG $\mathcal{Q}$-coalgebra), then $\mathfrak{C}\otimes C$ is cyclic as a DG coalgebra over $\mathcal{P}\otimes_{\H}\mathcal{Q}$. By \cite[Propositon 7.6.5]{LV}, there is a morphism of operads $\mathtt{Lie} \rar \mathcal{P}\otimes_{\H}\mathcal{Q}$. Let ${\mathcal Lie}^c({\mathfrak{C}} \otimes C)$ denote $\mathfrak{C} \otimes C$ viewed as a Lie coalgebra. It follows that ${\mathcal Lie}^c({\mathfrak{C}} \otimes C)$ is a cyclic Lie coalgebra. By \cite[Lemma 5.1]{BRZ},
$\cb_{\mathtt{Lie}}[{\mathcal Lie}^c({\mathfrak{C}} \otimes C)]$ acquires a DG Poisson structure. Since $\cb_{\mathtt{Lie}}[{\mathcal Lie}^c({\mathfrak{C}} \otimes C)]$ is a cofibrant representative of $\DRep_{\mathfrak{C}}(A)$ by \cite[Thm. A.4]{BFPRW}, the desired proposition follows.
\eproof

\subsection{Derived Poisson structures on algebras over operads}

Let $A$ be a DG $\mathcal{P}$-algebra. Let $\Der_{\mathcal{P}}(A)$ denote the DG Lie algebra of derivations from $A$ to itself. Then $A$ and therefore, $A \otimes A$ are DG Lie modules over $\Der_{\mathcal{P}}(A)$. It is not difficult to verify that the $\Der_{\mathcal{P}}(A)$-module structure on $A \otimes A$ descends to a $\Der_{\mathcal{P}}(A)$-module structure on $A_\n$. Following Crawley-Boevey \cite{CB}, we define a {\it Poisson structure} on $A$ to be a Lie bracket $\{ \mbox{--},\mbox{--}\}$ on $A_\n$ such that for all $\alpha\,\in\,A_\n$, the operator $\{\alpha,\mbox{--}\}\,:\,A_\n \rar A_\n$ is induced by a derivation $\partial_{\alpha}\,\in\,\Der_{\mathcal{P}}(A)$. A {\it derived Poisson structure} on $A$ is defined to be a Poisson structure on some cofibrant resolution of $A$.

For the rest of this section, we further assume that the operad $\mathcal{P}$ is cyclic. Recall that in this case, there is an action of the cyclic permutation $\tau$ of order $n+1$ on $\mathcal{P}(n)$ for each $n \geqslant 0$. By \cite[13.14.6]{LV}, $\tau$ may be thought of as changing the last input into the output and the output into the first input. The following formulas relate the $\tau$-action with the partial compositions of $\mathcal{P}$:
\begin{eqnarray*}
\tau(\mu\circ_{i}\nu) &=& \tau(\mu)\circ_{i+1}\nu,\ \ \ \ 1\leqslant i <n\,,\\
\tau(\mu\circ_{n}\nu) &=& \tau(\nu)\circ_{1}\tau(\mu)\,,
\end{eqnarray*}
where $\mu\in\mathcal{P}(n)$ and $\nu\in\mathcal{P}(m)$. The reader may refer to \cite[Fig. 13.2]{LV} for the illustration of the formulas above. For fixed $n$, we write $N=\sum_{i=0}^{n}\tau^{i}$.

Let $A$ be Koszul dual to a {\it cyclic} (conilpotent) DG $\mathcal{Q}$-coalgebra $C$. Let $\mathfrak{C}$ be a finite dimensional $\mathcal{P}$-coalgebra in degree $0$ equipped with an invariant $2$-tensor $\mathrm{coTr}\,:\,k \rar \mathfrak{C}^\n$. Let $\mathscr S\,:=\,\mathfrak{C}^{\ast}$. Let $\Tr\,:\,\mathscr S_{\n} \rar k$ denote the dual of $\mathrm{coTr}$. Assume that the pairing
$$ \begin{diagram} \mathscr{S} \otimes \mathscr{S} & \rTo & \mathscr{S}_\n & \rTo^{\Tr} & k \end{diagram} $$
is non-degenerate. Since the above pairing is cyclic on $\mathscr{S}$ by definition, it induces a cyclic pairing on $\mathfrak{C}$ as well. By Proposition \ref{OpLCBPoiss}, $\HR_\bullet(A,\mathscr S):= \HR_\bullet(A,\mathfrak{C})$ acquires a graded Poisson structure. The following generalization of Theorem \ref{t3s2int} and Theorem \ref{liehomom} is the main result of this section.

\begin{theorem} \la{operadpoiss}
There is a derived Poisson structure on $A$ such that $\Tr\,:\,\HC_\bullet(\mathcal{P},A) \rar \HR_\bullet(A,\mathscr S)$ is a graded Lie algebra homomorphism.
\end{theorem}

\subsubsection{Proof of Theorem \ref{operadpoiss}}

Let $V:=C[-1]$. There is an isomorphism of $\mathcal{P}$-algebras $\cb_{\mathcal{Q}}(C)\,\cong\,T_{\mathcal{P}}V$, where $T_{\mathcal{P}}V$ is the free $\mathcal{P}$-algebra generated by $V$. For notational brevity, let $R:=T_{\mathcal{P}}V$. We define the {\it cyclic derivative} $\bar{\partial}\,:\,R_{\n}\rar R \otimes V$ as follows. On the direct summand $\mathcal{P}(n)\otimes_{\mathbb{S}_{n+1}}V^{(n+1)}$ of $R_{\n}$, it is given by
\begin{equation}\label{Opcyclicder}
\begin{diagram}
\mathcal{P}(n)\otimes_{\mathbb{S}_{n+1}}V^{(n+1)} & \rTo^{N\cdot(\mbox{--})} & \mathcal{P}(n)\otimes_{\mathbb{S}_{n+1}}V^{(n+1)}  & \rTo & \mathcal{P}(n)\otimes_{\mathbb{S}_{n}}V^{n}\otimes V\,,
\end{diagram}
\end{equation}
where $N$ acts diagonally and where the last arrow is the obvious map that permutes no factors. Explicitly, given $\mu\otimes v_{1}\otimes \cdots \otimes v_{n}\otimes v_{n+1}\in R_{\n}$,
\begin{eqnarray}\label{Opcyclicder2}
\bar{\partial}(\mu\otimes v_{1}\otimes \cdots \otimes v_{n}\otimes v_{n+1}) &=&\sum_{i=1}^{n+1}(\tau^{n+1-i}\mu)\otimes \tau^{n+1-i}(v_{1}\otimes \cdots \otimes v_{n}\otimes v_{n+1})\nonumber\\
&=& \sum_{i=1}^{n+1}\pm(\tau^{n+1-i}\mu)\otimes(v_{i+1}\otimes \cdots \otimes v_{n+1}\otimes v_{1}\cdots \otimes v_{i-1}\cdots v_{i})\nonumber\\
&=& \sum_{i=1}^{n+1}\pm(\tau^{n+1-i}\mu)(v_{i+1}\otimes \cdots \otimes v_{n+1}\otimes v_{1}\cdots \otimes v_{i-1})\otimes v_{i} \ .
\end{eqnarray}

Note that the cyclic pairing on $C$ gives a skew-symmetric pairing on $V$. Following \cite[Sec. 6]{G}, define a bilinear bracket $\{ \mbox{--}\,,\,\mbox{--} \}\,:\,R_{\n}\otimes R_{\n}\rar R_{\n}$ as the following composition
\begin{equation}\label{OpcbLie1}
\begin{diagram}
R_{\n}\otimes R_{\n} & \rTo^{\bar{\partial}\otimes \bar{\partial}} & (R\otimes V)^{\otimes 2} & \rTo^{\sim} & R\otimes R \otimes V \otimes V & \rTo^{\Id_{R\otimes R}\otimes \langle \mbox{--}\,,\,\mbox{--} \rangle} & R\otimes R & \rTo & R_{\n}\,,
\end{diagram}
\end{equation}
where the last arrow is the canonical projection from $R\otimes R$ to $R_{\n}$. By Lemma \ref{OpcbLie} below, the above bracket defines a Poisson structure on $\cb_{\mathcal{Q}}(C)$: the compatibility of the Lie bracket on $R_{\n}$ constructed therein with the differential on $\cb_{\mathcal{Q}}(C)_{\n}$ follows from the cyclicity of $C$. By definition, this gives $A$ a derived Poisson structure.

\begin{lemma}\label{OpcbLie}
The bracket \eqref{OpcbLie1} defines a Poisson structure on $R$. In particular, it is a Lie bracket on $R_{\n}$.
\end{lemma}

\begin{proof}
Since the canonical projection $R \otimes R \twoheadrightarrow R_\n$ is symmetric, the bracket $\{\mbox{--},\mbox{--}\}$ on $R_{\n}$ is skew-symmetric.
Define an action\footnote{At this stage, we have {\it not} proven any fact about the structure of $R_{\n}$. The term `action' of $R_{\n}$ on a vector space $V$ simply means a linear map from $R \otimes V \rar V$.} of $R_{\n}$ on $R$ via
\begin{equation*}
\{\alpha,\,x\}=\sum_{i=1}^{n+1}\sum_{j=1}^{m}\pm\langle v_{i},\,w_{j}\rangle (\nu\circ_{j}(\tau^{n+1-i}\mu))\otimes (w_{1}\otimes \cdots \otimes w_{j-1}\otimes v_{i+1}\otimes\cdots\otimes v_{i-1}\otimes w_{j+1}\otimes \cdots\otimes w_{m})\,.
\end{equation*}
where $\alpha=\mu\otimes v_{1}\otimes \cdots \otimes v_{n}\otimes v_{n+1}\in R_{\n}$ and $x=\nu\otimes w_{1}\otimes \cdots \otimes w_{m}\in R$. By definition, the endomorphism $\{\alpha,\mbox{--}\}\,:\, R \rar R$ is a graded derivation for any (homogeneous) $\alpha\,\in\,R_{\n}$. The above action of $R_{\n}$ on $R$ extends to an action of $R_{\n}$ on $R \otimes R$ via the Leibniz rule.

To prove that the bracket \eqref{OpcbLie1} satisfies the Jacobi identity, we begin by showing that the canonical projection $p: R \otimes R \rar R_{\n}$ is $R_{\n}$-equivariant,  i.e., given $\alpha=\mu\otimes v_{1}\otimes \cdots \otimes v_{n}\otimes v_{n+1}\in R_{\n}$, $x=\mu'\otimes w_{1}\otimes \cdots \otimes w_{m},\,y=\mu''\otimes u_{1}\otimes \cdots \otimes u_{l}\in R$, we want to show that
\begin{equation}\label{Uninveq}
\{\alpha,\,p(x\otimes y)\}=p\big(\{\alpha,\,x\otimes y\}\big)=p\big(\{\alpha,\,x\}\otimes y+\pm x\otimes \{\alpha,\,y\}\big)\,,
\end{equation}
Notice that
\begin{eqnarray*}
p(x\otimes y)&=&\pm(\tau^{-1}\mu')\otimes w_{2}\otimes\cdots\otimes w_{m}\otimes\mu''(u_{1},\,\cdots,\,u_{l})\otimes w_{1}\\
&=&\pm((\tau^{-1}\mu')\circ_{n}\mu'')\otimes (w_{2}\otimes\cdots\otimes w_{m}\otimes u_{1}\otimes \cdots\otimes u_{l}\otimes w_{1})\,.
\end{eqnarray*}

Hence
\begin{eqnarray}
\{\alpha,\,p(x\otimes y)\} &=& \sum_{i=1}^{n+1}\sum_{j=1}^{m}\pm\langle v_{i},\,w_{j}\rangle ((\tau^{n-i}\mu)\circ_{n}(\tau^{m+1+l-j}((\tau^{-1}\mu')\circ_{n}\mu'')))\otimes\nonumber\\
&&(v_{i+2}\otimes \cdots \otimes v_{i-1}\otimes w_{j+1}\otimes\cdots\otimes w_{m}\otimes u_{1}\otimes \cdots\otimes u_{l}\otimes w_{1}\otimes \cdots \otimes w_{j-1}\otimes v_{i+1})\nonumber\\
&+&\sum_{i=1}^{n+1}\sum_{k=1}^{l}\pm\langle v_{i},\,u_{k}\rangle((\tau^{n-i}\mu)\circ_{n}(\tau^{l+1-k}((\tau^{-1}\mu')\circ_{n}\mu'')))\otimes\nonumber\\
&&(v_{i+2}\otimes \cdots \otimes v_{i-1}\otimes u_{k+1}\otimes\cdots\otimes u_{l}\otimes w_{1}\otimes \cdots \otimes w_{m}\otimes u_{1}\otimes \cdots\otimes u_{k-1}\otimes v_{i+1})\nonumber\\
&=& \sum_{i=1}^{n+1}\sum_{j=1}^{m}\pm\langle v_{i},\,w_{j}\rangle ((\tau^{n-i}\mu)\circ_{n}((\tau^{m+1-j}\mu')\circ_{m+1-j}\mu''))\otimes\nonumber\\
&&(v_{i+2}\otimes \cdots \otimes v_{i-1}\otimes w_{j+1}\otimes\cdots\otimes w_{m}\otimes u_{1}\otimes \cdots\otimes u_{l}\otimes w_{1}\otimes \cdots \otimes w_{j-1}\otimes v_{i+1})\label{Uninveq1}\\
&+&\sum_{i=1}^{n+1}\sum_{k=1}^{l}\pm\langle v_{i},\,u_{k}\rangle((\tau^{n-i}\mu)\circ_{n}((\tau^{l+1-k}\mu'')\circ_{l+1-k}\mu'))\otimes\nonumber\\
&&(v_{i+2}\otimes \cdots \otimes v_{i-1}\otimes u_{k+1}\otimes\cdots\otimes u_{l}\otimes w_{1}\otimes \cdots \otimes w_{m}\otimes u_{1}\otimes \cdots\otimes u_{k-1}\otimes v_{i+1})\label{Uninveq2}
\end{eqnarray}

On the other hand,
\begin{eqnarray}
p\big(\{\alpha,\,x\}\otimes y\big)&=&\sum_{i=1}^{n+1}\sum_{j=1}^{m}\pm\langle v_{i},\,w_{j}\rangle((\tau^{-1}(\mu'\circ_{j}(\tau^{n+1-i}\mu)))\circ_{n+m-1}\mu'')\otimes\nonumber\\
&&(w_{2}\otimes \cdots\otimes w_{j-1}\otimes v_{i+1}\otimes\cdots\otimes v_{i-1}\otimes w_{j+1}\otimes \cdots\otimes w_{m}\otimes u_{1}\otimes \cdots\otimes u_{l}\otimes w_{1})\nonumber\\
&=&\sum_{i=1}^{n+1}\sum_{j=1}^{m}\pm\langle v_{i},\,w_{j}\rangle(\tau^{n+m-j}(\mu'\circ_{j}(\tau^{n+1-i}\mu)))\circ_{n+m-j}\mu'')\otimes\nonumber\\
&&(v_{i+2}\otimes \cdots \otimes v_{i-1}\otimes w_{j+1}\otimes\cdots\otimes w_{m}\otimes u_{1}\otimes \cdots\otimes u_{l}\otimes w_{1}\otimes \cdots \otimes w_{j-1}\otimes v_{i+1})\label{Uninveq3}
\end{eqnarray}
Similarly
\begin{eqnarray}
\pm p\big(x\otimes \{\alpha,\,y\}\big)&=&\sum_{i=1}^{n+1}\sum_{k=1}^{l}\pm\langle v_{i},\,u_{k}\rangle((\tau^{-1}\mu')\circ_{m}((\mu''\circ_{k}(\tau^{n+1-i}\mu)))\otimes\nonumber\\
&&(w_{2}\otimes\cdots\otimes w_{m}\otimes u_{1}\otimes\cdots\otimes u_{k-1}\otimes v_{i+1}\otimes \cdots \otimes v_{i-1}\otimes u_{k+1}\otimes \cdots \otimes u_{l}\otimes w_{1})\nonumber\\
&=&\sum_{i=1}^{n+1}\sum_{k=1}^{l}\pm\langle v_{i},\,u_{k}\rangle((\tau^{n+l-k}(\mu''\circ_{k}(\tau^{n+1-i}\mu)))\circ_{n+l-k}\mu')\otimes\nonumber\\
&&(v_{i+2}\otimes \cdots \otimes v_{i-1}\otimes u_{k+1}\otimes\cdots\otimes u_{l}\otimes w_{1}\otimes \cdots \otimes w_{m}\otimes u_{1}\otimes \cdots\otimes u_{k-1}\otimes v_{i+1})\label{Uninveq4}
\end{eqnarray}

It is easy to see that (\ref{Uninveq1}) equals (\ref{Uninveq3}) and (\ref{Uninveq2}) equals  (\ref{Uninveq4}). This verifies \eqref{Uninveq}.

Next, we show that given $\alpha=\mu\otimes v_{1}\otimes \cdots \otimes v_{n}\otimes v_{n+1}, \,\beta=\nu\otimes w_{1}\otimes \cdots \otimes w_{m}\otimes w_{m+1}\in R_{\n}$ and $u\in V\subseteq R$,
\begin{equation}\label{OpcbLieAct1}
\{\{\alpha,\,\beta\},\,u\}=\{\alpha,\,\{\beta,\,u\}\}-\pm\{\beta,\,\{\alpha,\,u\}\}\ .
\end{equation}

Indeed, by (\ref{OpcbLie1}),
\begin{eqnarray*}
\{\alpha,\,\beta\}&=& \sum_{i=1}^{n+1}\sum_{j=1}^{m+1} \pm\langle v_{i},\,w_{j}\rangle(\tau^{n+1-i}\mu)(v_{i+1}\otimes \cdots \otimes v_{i-1})\otimes (\tau^{m+1-j}\nu)(w_{j+1}\otimes \cdots \otimes w_{j-1})\\
&=&\sum_{i=1}^{n+1}\sum_{j=1}^{m+1} \pm\langle v_{i},\,w_{j}\rangle((\tau^{n-i}\mu)\circ_{n}(\tau^{m+1-j}\nu))\otimes(v_{i+2}\otimes \cdots \otimes v_{i-1}\otimes w_{j+1}\otimes \cdots \otimes w_{j-1}\otimes v_{i+1})\,.
\end{eqnarray*}
Hence
\begin{eqnarray}
\{\{\alpha,\,\beta\},\,u\} &=&\sum_{i=1}^{n+1}\sum_{j=1}^{m+1}\sum_{k=1}^{j-1}\pm\langle v_{i},\,w_{j}\rangle\langle w_{k},\,u\rangle\tau^{j-k}((\tau^{n-i}\mu)\circ_{n}(\tau^{m+1-j}\nu))\otimes\nonumber\\
&&(w_{k+1}\otimes \cdots \otimes w_{j-1}\otimes v_{i+1}\otimes \cdots \otimes v_{i-1}\otimes w_{j+1}\otimes \cdots \otimes w_{k-1})\label{OpcbLieAct2}\\
&+& \sum_{i=1}^{n+1}\sum_{j=1}^{m+1}\sum_{k=j+1}^{m+1}\pm\langle v_{i},\,w_{j}\rangle\langle w_{k},\,u\rangle\tau^{m+1+j-k}((\tau^{n-i}\mu)\circ_{n}(\tau^{m+1-j}\nu))\otimes \nonumber\\
&&(w_{k+1}\otimes \cdots \otimes w_{j-1}\otimes v_{i+1}\otimes \cdots \otimes v_{i-1}\otimes w_{j+1}\otimes \cdots \otimes w_{k-1})\label{OpcbLieAct3}\\
&+& \sum_{i=1}^{n+1}\sum_{j=1}^{m+1}\sum_{k=1}^{i-1}\pm\langle v_{i},\,w_{j}\rangle\langle v_{k},\,u\rangle\tau^{m+i-k}((\tau^{n-i}\mu)\circ_{n}(\tau^{m+1-j}\nu))\otimes \nonumber\\
&&(v_{k+1}\otimes \cdots \otimes v_{i-1}\otimes w_{j+1}\otimes \cdots \otimes w_{j-1}\otimes v_{i+1}\otimes \cdots \otimes v_{k-1})\label{OpcbLieAct4}\\
&+& \sum_{i=1}^{n+1}\sum_{j=1}^{m+1}\sum_{k=i+1}^{n+1}\pm\langle v_{i},\,w_{j}\rangle\langle v_{k},\,u\rangle\tau^{m+n+1+i-k}((\tau^{n-i}\mu)\circ_{n}(\tau^{m+1-j}\nu))\otimes  \nonumber\\
&&(v_{k+1}\otimes \cdots \otimes v_{i-1}\otimes w_{j+1}\otimes \cdots \otimes w_{j-1}\otimes v_{i+1}\otimes \cdots \otimes v_{k-1})\label{OpcbLieAct5}
\end{eqnarray}

\begin{equation*}
\{\beta,\,u\}=\sum_{k=1}^{m+1}\pm\langle w_{k},\,u\rangle(\tau^{m+1-k}\nu)(w_{k+1}\otimes \cdots \otimes w_{k-1})\,.
\end{equation*}
Hence
\begin{eqnarray}
\{\alpha,\,\{\beta,\,u\}\} &=& \sum_{i=1}^{n+1}\sum_{j=1}^{k-1}\sum_{k=1}^{m+1}\pm\langle v_{i},\,w_{j}\rangle\langle w_{k},\,u\rangle((\tau^{m+1-k}\nu)\circ_{m+1+j-k}(\tau^{n+1-i}\mu))\otimes\nonumber\\
&&(w_{k+1}\otimes \cdots \otimes w_{j-1}\otimes v_{i+1}\otimes \cdots \otimes v_{i-1}\otimes w_{j+1}\otimes \cdots \otimes w_{k-1})\label{OpcbLieAct6}\\
&+&\sum_{i=1}^{n+1}\sum_{j=k+1}^{m+1}\sum_{k=1}^{m+1}\pm\langle v_{i},\,w_{j}\rangle\langle w_{k},\,u\rangle((\tau^{m+1-k}\nu)\circ_{j-k}(\tau^{n+1-i}\mu))\otimes\nonumber\\
&&(w_{k+1}\otimes \cdots \otimes w_{j-1}\otimes v_{i+1}\otimes \cdots \otimes v_{i-1}\otimes w_{j+1}\otimes \cdots \otimes w_{k-1})\label{OpcbLieAct7}
\end{eqnarray}
Similarly,
\begin{eqnarray}
\{\beta,\,\{\alpha,\,u\}\} &=& \sum_{j=1}^{m+1}\sum_{i=1}^{k-1}\sum_{k=1}^{n+1}\pm\langle w_{j},\,v_{i}\rangle\langle v_{k},\,u\rangle((\tau^{n+1-k}\mu)\circ_{n+1+i-k}(\tau^{m+1-j}\nu))\otimes\nonumber\\
&&(v_{k+1}\otimes \cdots \otimes v_{i-1}\otimes w_{j+1}\otimes \cdots \otimes w_{j-1}\otimes v_{i+1}\otimes \cdots \otimes v_{k-1})\label{OpcbLieAct8}\\
&+&\sum_{j=1}^{m+1}\sum_{i=k+1}^{n+1}\sum_{k=1}^{n+1}\pm\langle w_{j},\,v_{i}\rangle\langle v_{k},\,u\rangle((\tau^{n+1-k}\mu)\circ_{i-k}(\tau^{m+1-j}\nu))\otimes\nonumber\\
&&(v_{k+1}\otimes \cdots \otimes v_{i-1}\otimes w_{j+1}\otimes \cdots \otimes w_{j-1}\otimes v_{i+1}\otimes \cdots \otimes v_{k-1})\label{OpcbLieAct9}
\end{eqnarray}
One can check that (\ref{OpcbLieAct2}) equals  (\ref{OpcbLieAct7}), (\ref{OpcbLieAct3}) equals (\ref{OpcbLieAct6}), (\ref{OpcbLieAct4}) is identical to (\ref{OpcbLieAct9}), and (\ref{OpcbLieAct5}) equals (\ref{OpcbLieAct8}). This verifies (\ref{OpcbLieAct1}). Since $R_{\n}$ acts on $R$ by derivations,
$$ \{\{\alpha,\,\beta\},\,x\}=\{\alpha,\,\{\beta,\,x\}\}-\pm\{\beta,\,\{\alpha,\,x\}\} $$
for any $x \,\in\,R$. Therefore, for $\alpha,\,\beta \in R_{\n}$ and $x\otimes y\in R\otimes R$, we have
\begin{equation}\label{OpcbLieAct10}
\{\{\alpha,\,\beta\},\,x\otimes y\}=\{\alpha,\,\{\beta,\,x\otimes y\}\}-\pm\{\beta,\,\{\alpha,\,x\otimes y\}\}\ .
\end{equation}
Since $p:R \otimes R \rar R_{\n}$ is $R_{\n}$-equivariant,
\begin{equation}\label{OpcbLieAct11}
\{\{\alpha,\,\beta\},\,\gamma\}=\{\alpha,\,\{\beta,\,\gamma\}\}-\pm\{\beta,\,\{\alpha,\,\gamma\}\}\,
\end{equation}
for any $\gamma\,\in\,R_{\n}$. This verifies the Jacobi identity for the bracket (\ref{OpcbLie1}), proving that it is a Lie bracket. Since this bracket descends from an action of $R_{\n}$ on $R \otimes R$ induced by an action of $R_{\n}$ on $R$ by derivations (recall the $R_{\n}$-equivariance of $p$), this bracket defines a Poisson structure on $R$.
\end{proof}

It remains to show that $\Tr\,:\,R_{\n} \rar \mathfrak{C} \ltimes R$ is a DG Lie algebra homomorphism. By \cite[Theorem A.2]{BFPRW}, $C\ltimes R\cong \cb_{\mathtt{Lie}}[{\mathcal Lie}^c({\mathfrak{C}} \otimes C)]$. By Proposition \ref{OpLCBPoiss}, $\mathfrak{C}\ltimes R$ acquires a structure of DG Poisson (and therefore, Lie) algebra. Hence there is an action of $R_{\n}$ on $\mathfrak{C}\ltimes R$ via $\Tr$. Since $\mathfrak{C}\ltimes R$ is freely generated by $\mathfrak{C}\otimes V$ as a graded commutative algebra, $\Omega^{1}({\mathfrak{C}\ltimes R})\cong (\mathfrak{C}\ltimes R)\otimes \mathfrak{C}\otimes V$. Since $\mathfrak{C}$ is finite dimensional, $\mathbf{Hom}(\mathfrak{C}, B)\,\cong\,B \otimes \mathscr{S}$ for any $B\,\in\,\cDGA_{k/k}$. Let $\pi_R$ be the universal representation
$$\begin{diagram} R & \rTo & \mathbf{Hom}_{\ast}(\mathfrak{C}, \mathfrak{C} \ltimes R) & \rInto & \mathbf{Hom}(\mathfrak{C}, \mathfrak{C} \ltimes R)\,\cong\,(\mathfrak{C}\ltimes R)\otimes \mathscr{S}\end{diagram} \ . $$
 The following lemma generalizes Lemma \ref{lderham}. We leave its proof (which is very similar to that of Lemma \ref{lderham}) to the interested reader.

\blemma \la{OpderDR}
The following diagram commutes:
$$
\begin{diagram}
R_{\n} & \rTo^{\bar{\partial}} & R\otimes V & \rTo^{\pi_{R}\otimes \id} & (\mathfrak{C}\ltimes R)\otimes \mathscr{S} \otimes V\\
& \rdTo_{\Tr} & & & \dTo^{\cong}\\
& & \mathfrak{C}\ltimes R & \rTo^{d} & \Omega^1({\mathfrak{C}\ltimes R})
\end{diagram}
$$
Here, the vertical isomorphism on the right identifies $\mathscr{S}$ with $\mathfrak{C}$ via the pairing on $\mathscr{S}$.
\elemma
The following proposition follows from Lemma \ref{OpderDR} just as Proposition \ref{univrep} follows from Lemma \ref{lderham}.
\bprop \la{univrepOp}
The universal representation $\pi_R\,:\,R \rar (\mathfrak{C} \ltimes R) \otimes \mathscr{S}$ is $R_{\n}$-equivariant, where $R_{\n}$ acts trivially on $\mathscr{S}$.
\eprop
It follows from Proposition \ref{univrepOp} that the composite map
$$ \begin{diagram} R \otimes R & \rTo^{\pi_R \otimes \pi_R} & (\mathfrak{C} \ltimes R) \otimes \mathscr{S} \otimes (\mathfrak{C} \ltimes R) \otimes \mathscr{S}\,\cong\, (\mathfrak{C} \ltimes R)^{\otimes 2} \otimes \mathscr{S}^{\otimes 2} & \rTo^{\mu \otimes \langle \mbox{--},\mbox{--}\rangle} & \mathfrak{C} \ltimes R \end{diagram} $$
is $R_{\n}$-equivariant. Denote the above map by $\Tr_{R \otimes R}$. Since the diagram below commutes,
$$ \begin{diagram}
R \otimes R & \rTo^{\Tr_{R \otimes R}} & \mathfrak{C} \ltimes R\\
  \dTo^{p} & \ruTo^{\Tr} & \\
  R_{\n} & &
  \end{diagram}$$
and since $p$ is $R_{\n}$-equivariant as well, $\Tr$ is $R_{\n}$-equivariant. This proves that $\Tr:R_{\n} \rar \mathfrak{C} \ltimes L$ is a DG Lie algebra homomorphism. The second assertion of Theorem \ref{operadpoiss} follows on homologies.

\end{document}